\documentclass[conference]{IEEEtran}
\IEEEoverridecommandlockouts

\usepackage{graphicx}
\graphicspath{{images/}}
\usepackage{subfig}
\usepackage{cite}
\usepackage{algorithm}
\usepackage{amsmath,amssymb,amsfonts,amsthm}
\usepackage{algpseudocode}
\usepackage{textcomp}
\usepackage{xcolor}
\usepackage{booktabs}
\usepackage{cleveref}
\usepackage{url} 
\def\BibTeX{{\rm B\kern-.05em{\sc i\kern-.025em b}\kern-.08em
    T\kern-.1667em\lower.7ex\hbox{E}\kern-.125emX}}

\newcommand\lambdagap{\lambda_{\text{gap}}}
\newtheorem{proposition}{Proposition}

\begin{document}

\title{Support Discovery With Iteratively Reweighted Least Squares for Fixed-Charge Network Flow\\}
\author{\IEEEauthorblockN{Sindura Saraswathi}
\textit{University of Central Florida}\\
sindura@ucf.edu
\and
\IEEEauthorblockN{Christian Kümmerle}
\textit{University of Central Florida}\\

kuemmerle@ucf.edu}

\maketitle

\pagestyle{plain}
\thispagestyle{plain}

\begin{abstract} 
The fixed-charge network flow problem (FCNFP) couples continuous flow allocation with discrete arc-activation decisions, making it a canonical but computationally challenging model for a variety of network design and resource allocation problems. Exact mixed-integer linear programming formulations capture the fixed-charge structure faithfully, but often become difficult to solve on large networks. We propose a scalable continuous-optimization algorithm for large-scale single-commodity FCNFP based on an iteratively reweighted least-squares (IRLS) framework. The method replaces the discontinuous fixed-charge and linear arc cost objective with a smooth nonconvex Lasry--Lions surrogate and solves a sequence of weighted quadratic flow subproblems. Each subproblem is solved by a warm-started dual semismooth Newton method whose Newton systems have weighted graph-Laplacian structure, enabling the use of modern Laplacian solvers. To further improve the discovered arc supports of the challenging underlying combinatorial problem, we also develop an algorithmic variant that incorporates objective-driven perturbation restarts and an anchor-union restricted search that jointly leverages supports discovered by IRLS and by complementary FCNFP heuristics. Computational experiments on $410$ benchmark, synthetic, and large-scale instances show that our method obtains the best objective quality among the evaluated scalable FCNFP algorithms, with a mean gap of $1.316\%$ to a time-limited MILP reference and a win-or-tie rate of $90.0\%$ among the non-MILP methods.  The results indicate that combining smooth continuous optimization with support-level search is an effective strategy for producing high-quality feasible solutions to large-scale FCNFP. 
\end{abstract}


\section{Introduction}
\label{sec:introduction}
Many problems in AI and operations research couple a discrete activation decision with a continuous allocation: which resources to switch on, and how much to route through them. Large-scale LLM inference is a recent instance, where deploying a serving configuration incurs a setup cost while routing requests through it incurs load-dependent operating cost \cite{cheng2026fast}; the same fixed-plus-variable structure underlies classical network design and resource allocation. What makes such problems hard is not the continuous part, which is often polynomial once the active set is fixed, but the discrete structure that must be discovered.

The fixed-charge network flow problem (FCNFP) provides a canonical model for this structure. Given a directed network, the goal is to route flow satisfying supply-demand balance and capacity constraints while minimizing both proportional flow costs and fixed charges incurred when arcs are activated. FCNFP arises in transportation \cite{sadeghi2019new, hajiaghaei2014integrated, golmohamadi2017solving}, supply-chain design \cite{willet2021waterroute}, and infrastructure planning \cite{bogs2025planning}. FCNFP is NP-hard \cite{guisewite1990minimum}, with its computational difficulty arising from the discrete arc-activation structure. If the set of active arcs were known in advance, the remaining problem would reduce to a linear minimum cost network flow problem, which can be solved efficiently by linear programming methods or specialized solvers in $O(nm\log(n^2/m)\log(n\Gamma))$, where $n$ and $m$ denote the numbers of nodes and arcs and $\Gamma$ bounds the integer arc costs \cite{goldberg1990finding}. 
Exact mixed-integer linear programming (MILP) \cite{murty1968solving, gendron2014branch, ortega2003branch, palekar1990branch} formulations capture this structure faithfully, but they can become prohibitively expensive on large instances. In our experiments, a time-limited OR-Tools \cite{ortools} MILP baseline solved all instances up to $60$ nodes and $400$ arcs, but began failing to prove optimality at $120$ nodes and $1500$ arcs and failed to prove optimality on all $1000$-node, $20,000$-arc instances. 

This motivates a principled continuous optimization approach that preserves the computational advantages of network-flow structure while inducing sparse arc support. This perspective comes from sparse linear inverse problems, where the combinatorial task of support selection is often handled through sparsity-promoting continuous penalties and iterative reweighting schemes \cite{daubechies2004iterative, Daubechies-CPAM2010}.

In this paper, we propose a continuous optimization strategy for large-scale single-commodity FCNFP. The method is based on an iteratively reweighted least-squares (IRLS) \cite{Daubechies-CPAM2010, kummerle2021iteratively} framework that replaces the discontinuous fixed-plus-linear arc cost with a nonconvex surrogate smoothed by Lasry--Lions regularization \cite{simoes2021lasry}.  The smoothed surrogate \cite{Chen2012smoothing} helps stabilize the nonconvex optimization landscape induced by fixed-charge penalties. At each iteration, the method solves a weighted continuous flow subproblem and updates the weights so that arcs carrying small flow are increasingly penalized. This reweighting mechanism encourages small-flow arcs to vanish and guides the continuous flow towards a sparse, high-quality feasible solution without explicitly optimizing over binary activation variables.

Despite the Lasry--Lions smoothing strategy avoiding many suboptimal local minima, the underlying problem is still fundamentally nonconvex, and there are no guarantees available that the discovered support set is optimal. To find even better support sets of the network flow, we propose to augment the base IRLS method with two support-improvement layers. First, objective-driven perturbation restarts modify the incumbent support by suppressing expensive active arcs and opening promising inactive arcs, followed by projection back onto the feasible capacitated-flow polytope. Second, an anchor-union restricted search combines IRLS supports with supports discovered by external FCNFP heuristics, and then reruns IRLS on the resulting restricted subgraphs.

The main contribution of this work is a scalable primal algorithm that bridges smooth continuous optimization and support-level combinatorial search. Unlike exact MILP approaches, the proposed method does not attempt to certify global optimality. Instead, it is designed to produce high-quality feasible solutions on instances where exact methods are too slow or terminate with only feasible, but suboptimal flow vectors. Computational experiments on $410$ benchmark, synthetic, and large-scale instances show that our method achieves the best objective quality among the scalable evaluated FCNFP heuristics. It achieves a mean gap of $1.316\%$ to the MILP reference and, among the non-MILP methods, wins or ties on $90.0\%$ of the instances.

\section{Related Work}
\label{sec:related_work}
\paragraph{Existing FCNFP Solvers.}
FCNFP is a representative problem in the broader class of nonconvex minimum-cost flow problems and is NP-hard.
Exact MILP and branch-and-cut methods can model the FCNFP structure directly, but their scalability is limited by the binary arc-activation variables. This has motivated a large body of approximation and reformulation methods that exploit the underlying network-flow structure.

The dynamic slope scaling procedure (DSSP) \cite{kim1999solution} replaces the fixed-charge objective with a sequence of linear min-cost-flow (MCF) subproblems by converting each arc's fixed cost into a flow-dependent per-unit slope, producing an LP with network-flow constraints. These subproblems can be solved by generic LP solvers or by specialized min-cost-flow algorithms, with cost-scaling variants admitting the same $O(nm\log(n^2/m)\log(n\Gamma))$ bound, but the linearization only locally approximates the discontinuous term, making DSSP sensitive to initialization and prone to poor early support decisions. Related dynamic updating methods refine this idea. ADCUP \cite{nahapetyan2008adaptive} approximates FCNFP by a concave piecewise-linear network-flow problem and updates it dynamically, while \cite{rebennack2009bilinear} give a continuous bilinear formulation that represents the fixed-charge structure without binary variables, at the cost of nonconvexity; \cite{nie2020dynamic} build on this with a smoothing parameter and coupled linear programming subproblems. More recently, \cite{yang2024sequential} proposed the adaptive dynamic slope scaling procedure (ADSSP) with edge deletion, emphasizing the role of support identification and arc elimination in large-scale FCNFP.

Our method is related to these approaches in that it avoids full mixed-integer search and instead solves a sequence of continuous network-flow subproblems. However, our methodology fundamentally differs in several key aspects. First, rather than using slope scaling, which approximates the fixed-charge term with a sequence of linear surrogates, or an exact but nonconvex bilinear reformulation, we implicitly optimize a family of tailored, continuous surrogate objectives derived from a Lasry--Lions smoothing strategy \cite{LasryLions-1986,AttouchAze-1993Approximation,simoes2021lasry} that converge to the original FCNFP objective in the limit. By adjusting the smoothing parameters of this surrogate family iteratively, our method increases the degree of non-convexity over time, allowing for the discovery of high-quality supports. Furthermore, the efficiently solvable subproblems we use are not linear or bilinear, but quadratic, allowing for an appropriate balance between solver efficiency and modeling accuracy.

\paragraph{Smoothing Strategies and IRLS Framework.}
Smoothing has been developed successfully for a variety of nonsmooth optimization problems, especially nonsmooth convex composite optimization \cite{BenTalTeboulle-2006smoothing,Nesterov-2005SmoothMinimization,Beck2012smoothing}. Seen primarily as an algorithmic acceleration device for such problems, smoothing methods can also serve, in nonconvex settings, as a landscape regularization by exposing useful descent directions and delaying premature commitment to poor local supports \cite{Chen2012smoothing,Bian-2013WorstComplexitySmoothing,Yang-2020GraduatedNonConvexity,Kornowski2021oracle}. Specifically, the Lasry--Lions smoothing strategy \cite{LasryLions-1986,AttouchAze-1993Approximation,simoes2021lasry} that we adapt in this work is tailored to high-dimensional, nonconvex problems due to its flexibility involving two smoothing parameters, but does not directly lead to a computationally tractable optimization method \cite{Kornowski2021oracle}. We obtain a tractable algorithm by majorizing the Lasry--Lions surrogates by an adaptive quadratic model, which makes the proposed method an instance of the iteratively reweighted least-squares (IRLS) framework \cite{gorodnitsky_rao,Daubechies-CPAM2010,kummerle2021iteratively}. IRLS methods are well-known tools to solve hard instances of estimation problems with sparse combinatorial structure in computer vision \cite{ochs_dosovitskiy_brox_pock}, signal processing \cite{Daubechies-CPAM2010,wipf_nagarajan}, and machine learning \cite{KummerleMayrinkVerdun-ICML2021}, but have not been well-studied so far for network flow problems.

\paragraph{Efficient Laplacian Solvers.}
The scalability of our method is also tied to recent progress on Laplacian linear-system solvers. Nearly-linear-time algorithms for graph Laplacian and symmetric diagonally dominant systems have become a central tool in theoretical computer science \cite{Spielman-2004nearlypartitioning,Daitch-2008Faster,Teng-2010LaplacianParadigm,Spielman-2014NearlyLinear,Madry-2018ICM}, with best available runtime bounds of  $O(m (\log \log n)^{O(1)} \log(1/\epsilon))$ for finding an $\epsilon$-accurate solution \cite{Jambulapati2025ultrasparse}. While provably fast approximation algorithms have been developed for several network-flow problems \cite{christiano2011electrical, Daitch-2008Faster, karakostas2008faster}, implementable efficient algorithms have been so far only developed for Laplacian system solvers \cite{Kyng2016approximate,Gao-2023Robust,Gao2026ac}, which we leverage in this work.

\section{Problem Formulation}

Let $G=(V,E)$ be a directed network with node set $V$ and edge set $E$. 
Let $n=|V|$ and $m=|E|$. For each edge $e\in E$, let $f_e$ denote the flow on edge $e$, $c_e$ its capacity, $r_e$ its proportional cost, and $b_e$ its fixed cost. Let $C\in \mathbb{R}^{n\times m}$ denote the incidence matrix and let $d\in \mathbb{R}^n$ denote the demand vector, where $d_v>0$ at sink nodes, $d_v<0$ at source nodes, $d_v=0$ at transshipment nodes, with $\sum_v d_v=0$ so that total supply equals total demand. With these definitions, the fixed-charge network flow problem (FCNFP) can be formulated as
\begin{equation}\label{eq:jtrue}
\begin{split}
\min_{f} \quad & \sum_{e\in E} \left(r_e f_e + b_e \mathbf{1}\{|f_e|>0\} \right) =: J_{\mathrm{true}}(f) \\
\text{s.t.}\quad & Cf=d, 0\leq f_e \leq c_e, \qquad \forall e\in E. 
\end{split}
\end{equation}
The key difficulty of FCNFP lies in the combinatorial nature of the fixed-charge cost $b_e \mathbf{1}\{|f_e|>0\}$, which is active if and only if the binary variable $z_e : =  \mathbf{1}\{|f_e|>0\}$ is non-zero, corresponding to the opening of the edge $e$ for a non-zero flow. The set of feasible flow vectors of \eqref{eq:jtrue} is $\mathcal{F} = \left\{f\in \mathbb{R}^m:Cf=d,\;0\leq f \leq c \right\}$.


\subsection{Lasry--Lions Double Envelope Combined Surrogate}\label{sec:lasrylions}
The fixed-charge objective is discontinuous at zero, since an arbitrarily small positive flow activates the fixed charge, and it promotes sparsity, since once an arc carries flow, increasing that flow incurs no further charge. To obtain a tractable continuous formulation we replace the fixed-plus-linear arc cost with a smooth nonconvex surrogate based on the Lasry--Lions double envelope \cite{simoes2021lasry}, which removes the discontinuity while preserving the support-promoting structure.

The surrogate has three regimes and is governed by two smoothing parameters $\lambda>\mu>0$ shared across all edges. The parameter $\lambda$ controls the overall smoothing level, while $\mu$ controls the width of the transition region. Near zero, the surrogate is quadratic, which keeps the IRLS weights finite and the subproblem well defined. The transition region connects the quadratic and active-flow regimes smoothly. For larger flows, the exact envelope's linear tail has derivative $r_e$, so $b_e$ would enter only as an additive constant and vanish from the reweighting rule. We therefore replace the linear tail by a logarithmic term with coefficient $\alpha b_e$, where $\alpha=\lambda/(\lambda+s\lambda_0)$, $\lambda_0$ is the initial smoothing level and $s=10^{-2}$. 
 
For each edge $e$, the first transition point is $x_{1,e} = T_e(\lambda-\mu)/\lambda$, where $T_e=\lambda r_e+\sqrt{2\lambda b_e}$. The second transition point is $x_{2,e} = (S_e + \sqrt{S_e^2 - 4\mu \alpha b_e})/2$, where $S_e = T_e - \mu r_e$. We restrict $0<\mu<\lambda$ by the edge-wise cap $\mu\leq\min_e\mu_{\max,e}(\lambda)$, which makes these transitions real and ordered; the cap and ordering proof are given in supplementary material, \Cref{app:ll_derivation}, \Cref{app:fixed-parameter-majorization}. The resulting per-edge Lasry--Lions surrogate is
\begin{equation}\label{eq:lasrylions}
\mathrm{LL}_e^{\lambda,\mu}(x) =
\begin{cases}
\dfrac{x^2}{2(\lambda-\mu)},
\qquad |x|\leq x_{1,e}, \\[10pt]
\dfrac{T_e^2}{2\lambda} - \dfrac{(|x|-T_e)^2}{2\mu},
\quad x_{1,e}<|x|<x_{2,e}, \\[10pt]
\begin{aligned}[t]
&r_e(|x|-x_{2,e}) + \alpha b_e \log\left(\dfrac{|x|}{x_{2,e}}\right)\\
&\quad+ \psi_{2,e}(x_{2,e}),
\qquad |x|\geq x_{2,e};
\end{aligned}
\end{cases}
\end{equation}
where $\psi_{2,e}(x_{2,e})=T_e^2/(2\lambda)-(x_{2,e}-T_e)^2/(2\mu)$ makes the surrogate continuous at $x_{2,e}$. The transition points are chosen such that $\mathrm{LL}_e^{\lambda,\mu}$ is continuously differentiable across the three regimes (see supplementary material, \Cref{app:ll_derivation} for derivations). The full smoothed objective is
\begin{equation}\label{eq:jll}
J_{\mathrm{LL}}^{\lambda,\mu}(f) = \sum_{e\in E} \mathrm{LL}_e^{\lambda,\mu}(f_e).
\end{equation}

\Cref{fig:lasry_lions} illustrates this for an edge at three values of $\lambda$ with $\mu=\lambda/2$. For large $\lambda$, the quadratic and transition regions are wider and the discontinuity is smoothed more strongly. As $\lambda$ decreases, these regions contract toward the origin and the surrogate more closely tracks the fixed-plus-linear objective.

\begin{figure}[t]
    \centering
    \includegraphics[width=1\linewidth]{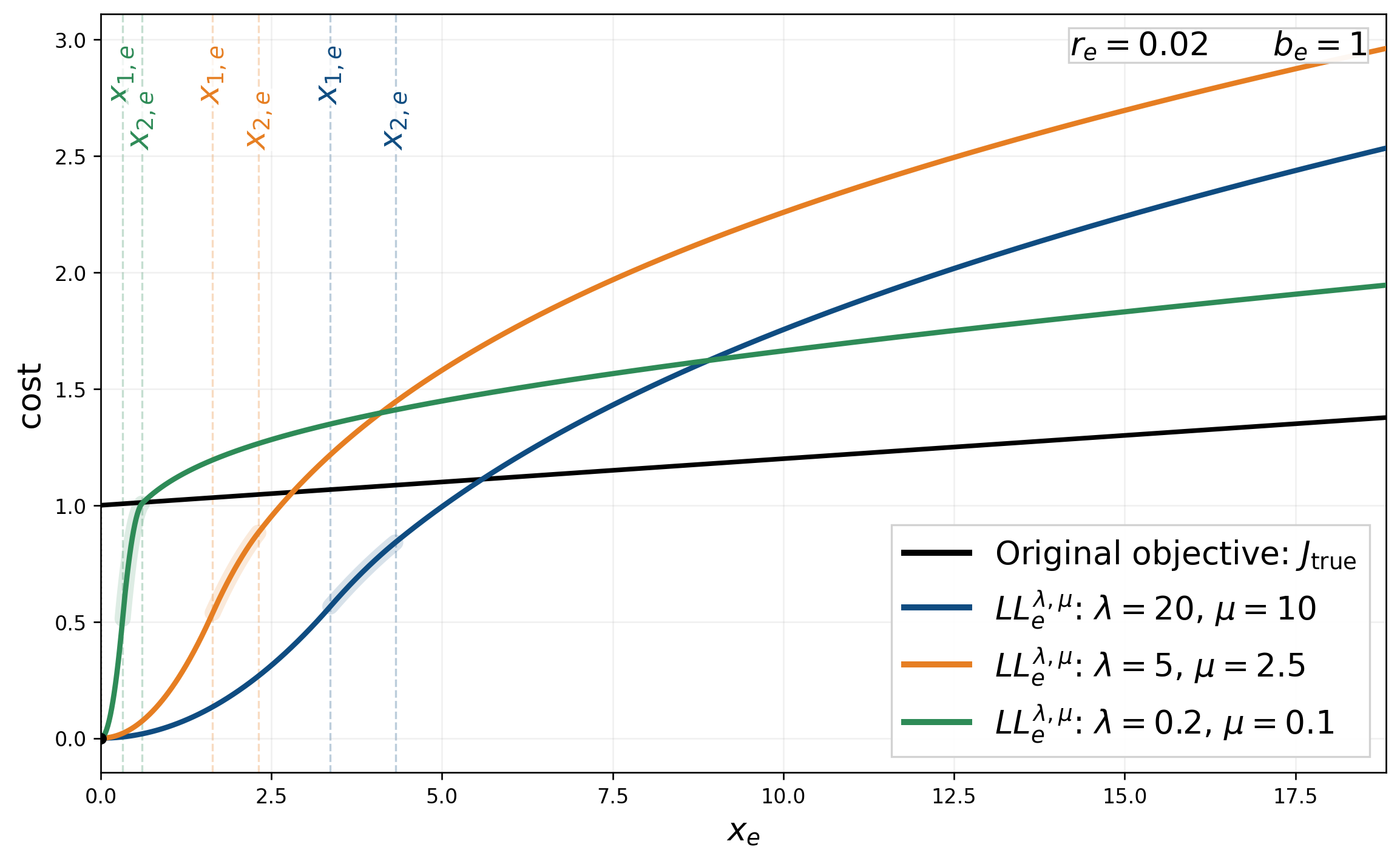}
    \caption{Lasry--Lions surrogate $\mathrm{LL}^{\lambda,\mu}_e$ for an edge with $r_e=0.02$, $b_e=1$, at three smoothing levels. $x_{1,e}$ and $x_{2,e}$ represent the transition points.}
    \label{fig:lasry_lions}
\end{figure}

\subsection{Iteratively Reweighted Least Squares}
\label{sec:irls}
We minimize the smoothed objective in \eqref{eq:jll} through an iteratively reweighted least-squares (IRLS) scheme \cite{kummerle2021iteratively}. The idea is to replace the nonconvex surrogate at each iteration by a convex weighted quadratic model whose weights are computed from the current flow.

At iteration $k$, given the current flow $f^{(k)}$, IRLS constructs edge weights $w_e^{(k)}$ and sets $W_k=\operatorname{diag}\left(w_1^{(k)},\ldots,w_m^{(k)}\right)$. With the weighted inner product $\langle u,v\rangle_{W_k}:=u^\top W_kv$, define the shifted quadratic model

\begin{equation}
\label{eq:mm-majorizer-appendix}
Q^{\lambda,\mu}(f\mid f^{(k)})=J_{\mathrm{LL}}^{\lambda,\mu}(f^{(k)})+\tfrac{1}{2}\langle f,f\rangle_{W_k}-\tfrac{1}{2}\langle f^{(k)},f^{(k)}\rangle_{W_k}.
\end{equation}
The terms independent of $f$ do not affect the minimizer. Hence the IRLS update and its equivalent weighted quadratic formulation can be written compactly as
\begin{align}
f^{(k+1)}
&=\operatorname*{argmin}_{f\in\mathcal F}
Q^{\lambda,\mu}(f\mid f^{(k)}) =\operatorname*{argmin}_{f\in\mathcal F}
\frac{1}{2}\langle f,f\rangle_{W_k}. \label{eq:newf}
\end{align}

The IRLS weights are obtained from the local curvature of the Lasry--Lions surrogate. For a nonzero flow value, we define the weight by differentiating the per-edge surrogate in \eqref{eq:lasrylions} and normalizing by the flow magnitude as $w_e(x)=\frac{\big(\mathrm{LL}_e^{\lambda,\mu}\big)'(|x|)}{|x|}$. This choice couples the quadratic subproblem to the local geometry of the Lasry--Lions surrogate. For parameters in the capped domain and exact analytical weights, $Q^{\lambda,\mu}(\,\cdot\mid f^{(k)})$ globally majorizes $J_{\mathrm{LL}}^{\lambda,\mu}$ and is tangent to it at the current iterate. Since small flows receive large weights, arcs carrying negligible flow become increasingly expensive in subsequent subproblems and are encouraged to vanish. Larger active flows receive smaller weights, allowing the algorithm to concentrate flow on a sparse support. The weights are given by
\begin{equation}\label{eq:weight_ll}
w_e(f_e) =
\begin{cases}
\dfrac{1}{\lambda-\mu},
& |f_e|\leq x_{1,e}, \\[10pt]
\dfrac{T_e-|f_e|}{\mu |f_e|},
& x_{1,e}<|f_e|<x_{2,e}, \\[12pt]
\dfrac{r_e}{|f_e|} + \dfrac{\alpha b_e}{|f_e|^2},
& |f_e|\geq x_{2,e}.
\end{cases}
\end{equation}
For $f_e=0$, the first case is used, so the weight remains finite.
In the active-flow region, the term $r_e/|f_e|$ represents the proportional cost contribution, while the term $\alpha b_e/|f_e|^2$ acts as an $\ell_0$-like fixed-charge penalty. 

Under the capped parameter condition, \Cref{prop:fixed-parameter-majorization} (Supplementary Material, \Cref{app:fixed-parameter-majorization}) shows that the exact weight and exact update in \eqref{eq:newf} constitute a majorization--minimization scheme \cite{SunBabuPalomar-IEEESP2017} for the fixed-parameter surrogate $J_{\mathrm{LL}}^{\lambda,\mu}$. At iteration $k$, the analytical weights are evaluated at the current flow $f^{(k)}$, giving $w_e^{(k)} = w_e(f_e^{(k)}),\ e\in E$. These weights are then fixed, and the next iterate is obtained by solving the weighted quadratic flow subproblem in \eqref{eq:newf}, which can be solved efficiently (see \Cref{sec:dual_ssn}). Each subproblem preserves the original flow-conservation and capacity constraints. The numerical implementation clips weights and solves QPs inexactly, so it verifies fixed-parameter surrogate descent directly rather than relying on exact majorization alone. Overall, the weight update and \eqref{eq:newf} provide a continuous mechanism for seeking low-cost FCNFP flows without explicitly introducing binary activation variables or tree-based search.

\subsection{Weighted Quadratic Program Solver}
\label{sec:dual_ssn}
Each IRLS iteration requires solving the weighted quadratic flow subproblem in \eqref{eq:newf}, which we solve by a warm-started dual semismooth Newton (SSN) method \cite{HintermuellerItoKunisch-2002}. Introducing a dual variable $y\in\mathbb{R}^n$ for the flow-conservation constraint $Cf=d$, the box-constrained primal minimizer for a fixed $y$ is available in closed form as $f(y)=\operatorname{clip}\!\left(W^{-1}C^\top y;\,0,\,c\right)$, where the clipping is applied component-wise. The corresponding dual residual is $R(y)=d-Cf(y)$. Thus, solving the subproblem is equivalent to finding a root of $R(y)=0$. An SSN step solves
\begin{equation}
\label{eq:ssnstep}
\left(CD(y)C^\top\right)\Delta y=d-Cf(y),
\end{equation}
where \(D(y)=\operatorname{diag}\!\left(w_e^{-1}\mathbf{1}\{0<(C^\top y)_e<w_ec_e\}\right)\). 
Only free arcs, those clipped at neither bound, contribute to this Newton system, so $C D(y) C^\top$ is a weighted graph Laplacian on the currently free subgraph, singular on each connected component's constant potential; we remove this gauge freedom by mean-centering. Because $D(y)$ vanishes on arcs at a bound, a literal Newton step from $y=0$ is degenerate; each solver call is instead seeded by a predictor read off the current primal flow, with a regularized, line-searched Newton correction when the predictor is rejected (see supplementary material, \Cref{app:dualssn}). Each Newton system is solved by preconditioned conjugate gradient (PCG) using a freshly computed approximate-Cholesky factorization \cite{githubGitHubAaiinstituteapproxchol} of the free-subgraph Laplacian as preconditioner, with an escalating iteration budget if the residual has not converged.
 
The solver is warm-started across IRLS iterations from the previous dual potential $y$ and the previous lower and upper active sets. Each Newton system costs $\widetilde O(m\log(1/\epsilon))$ \cite{Kyng2016approximate}, so a run of $K$ outer iterations averaging $\bar s$ Newton steps per subproblem performs $\widetilde O\!\left(K\,\bar s\,m\log(1/\epsilon)\right)$ work. In our runs, $\bar s$ stays below $4$ across a $200$-fold range in $m$ (see supplementary material, \Cref{app:dualssn}).

The IRLS weights are clipped to $[10^{-12},10^{12}]$, and each dual-SSN QP call is capped at $40$ Newton iterations. If the SSN does not converge and leaves a flow-conservation residual above $10^{-6}$, the algorithm falls back to the primal-dual interior-point method (see supplementary material, \Cref{app:dualssn}), which exploits the same weighted graph-Laplacian Newton systems.

\section{Algorithm}
\label{sec:algorithm}
We outline a core IRLS method for FCNFP in \Cref{alg:irls_core}: each iteration evaluates the weights \eqref{eq:weight_ll} at the current flow and solves the weighted quadratic subproblem over $\mathcal F$ with the dual semismooth Newton method of \Cref{sec:dual_ssn}. A trial iterate is accepted only if it does not increase $J_{\mathrm{LL}}^{\lambda,\mu}$ beyond a numerical tolerance; otherwise the algorithm backtracks along the segment joining the current and the trial flows; if no tested step is accepted, the current flow is retained. The reported solution is the best feasible candidate encountered during the run, scored under $J_{\mathrm{true}}$.

The run is organized into stages. Within a stage, $\lambda$ is updated non-increasingly based upon the flow gap quantity $\lambdagap$ (see \Cref{sec:lambda_update}). A new algorithmic stage is triggered if the relative flow change falls below $10^{-3}$, if a line search rejects the step or after six weighted-QP solves, in which case $\lambda$ is decreased more aggressively. 

Let $|f|_{(1)}\ge\cdots\ge|f|_{(m)}$ be the sorted flow magnitudes, let $K_{\text{gap}}=\arg\max_{1\le j<m}|f|_{(j)}/|f|_{(j+1)}$, and set $f_\uparrow=|f|_{(K_{\text{gap}})}$, $f_\downarrow=|f|_{(K_{\text{gap}}+1)}$, and $\rho_{\text{gap}}=f_\uparrow/f_\downarrow$. The pair $(f_\uparrow,f_\downarrow)$ brackets the widest gap in the flow magnitudes, so $K_{\text{gap}}$ is the support size that gap implies, $f_\downarrow$ the largest magnitude outside that support, and $\rho_{\text{gap}}$ a measure of how cleanly the two sides separate. This signal is used both by a $\lambda$-update rule and the support-refit procedure below. In practice, relevant arcs and a support estimate could be identified via an activation threshold $\tau$ (e.g., of $10^{-6}$), below which flows are set to zero; however, such a thresholding will yield generally infeasible flow vectors. To improve on this, we solve a linear min-cost flow restricted to the support set estimate $S$, which is updated whenever $\rho_{\text{gap}}\ge10$ and $0<K_{\text{gap}}\le\max\{4n,1024\}$ by setting $S=\{e:|f_e|>\sqrt{f_\uparrow f_\downarrow}\}$. While this is not an exact solver for the FCNFP problem on the $S$-restricted subgraph defined by $S$, this \emph{support refit procedure} \textsc{Refit} provides a mechanism to meaningfully score candidate supports. Each stage boundary additionally produces a full-graph candidate, which may open arcs outside $S$. Candidates enter $f^\star$ only when feasible and strictly improving.

\Cref{alg:irls_core} terminates if $\rho_{\text{gap}}\ge100$, $K_{\text{gap}}$ has been stable for four consecutive iterations, and $f^\star$ has not improved for eight iterations. Alternatively, termination is triggered if the objective value $J_{\mathrm{LL}}^{\lambda,\mu}(f^{(k)})$ stagnates or if $\lambda$ reaches a numerical precision floor. More details on support refit, step acceptance, stagnation, and candidate evaluation are given in the supplementary material, \Cref{app:algorithmic_details}.

\begin{algorithm}[t]
\caption{IRLS for FCNFP}
\label{alg:irls_core}
\begin{algorithmic}[1]
\Require Network $(C,d,c,r,b)$, $f^{(0)}\in\mathcal F$, $K=200$
\State $(\lambda_0,\mu_0)\gets\textsc{InitializeSmoothing}(f^{(0)})$;\; $f^\star\gets f^{(0)}$
\For{$k=0,\ldots,K-1$}
    \State $w^{(k)}\gets\textsc{Weights}(f^{(k)},\lambda_k,\mu_k)$,  \hfill{\small$\triangleright$ Eq.~\eqref{eq:weight_ll}} 
    \State $(\tilde f, y)\gets\textsc{DualSSN}(w^{(k)},\mathcal F; y)$ \hfill{\small$\triangleright$ Eq.~\eqref{eq:newf}}
    \State $(f^{(k+1)},y)\gets
    \textsc{Accept}(f^{(k)},\tilde f,\lambda_k,\mu_k)$ \hfill{\small$\triangleright$ backtrack}
    \State $g^{(k+1)}=(K_{\text{gap}},\rho_{\text{gap}},f_\uparrow,f_\downarrow)\gets\textsc{FlowGap}(f^{(k+1)})$
    \State $f^\star\gets\textsc{Best}\bigl(f^\star,f^{(k+1)},\textsc{Refit}(f^{(k+1)},g^{(k+1)})\bigr)$

    \State $(\lambda_{k+1},\mu_{k+1})
    \gets\textsc{Continuation}
    (\lambda_k,\mu_k,g^{(k+1)})$ 
    \If{$\textsc{Stop}$} \textbf{break} \EndIf \hfill{\small$\triangleright\ \lambda\le\lambda_{\min}$, stable gap, stagnation}
\EndFor
\State \Return $f^\star$
\end{algorithmic}
\end{algorithm}

\section{Smoothing Continuation Scheme}
\label{sec:lambda_update}

For the analytical surrogate with unmodified costs,
$J_{\mathrm{LL}}^{\lambda,\mu(\lambda)}$ converges pointwise to
$J_{\mathrm{true}}$ on $\mathcal F$ as $\lambda\downarrow0$ along any
admissible path $\mu(\lambda)$; in particular, this holds for
$\mu(\lambda)=\lambda/2$ (Supplementary Material, \Cref{app:ll_derivation}). This consistency
motivates starting from a broadly smoothed model and progressively sharpening
it toward the fixed-charge objective
\cite{peng2023convergence}. 
Thus, we decrease $\lambda$ by a stage-wise continuation rule and set $\mu=\min\{\lambda/2,\ \min_{e\in E}\mu_{\max,e}(\lambda)\}$ at every parameter update.
We propose a specific continuation schedule that balances rapid sharpening against premature support commitment. Specifically, since $T_e=\lambda r_e+\sqrt{2\lambda b_e}$ has the scale of a flow magnitude, $\lambda$ has units of squared flow per unit cost, which motivates fixing $\lambda_0$ at the scale of an initial feasible flow $f^{(0)}$. With $f_{\mathrm{med}}$ the median initial flow magnitude above one and $b_{\mathrm{med}}$ the median positive fixed cost, sparse flow initializations use $\lambda_0=10^{-2}\max\{f_{\mathrm{med}}^{2}/(2b_{\mathrm{med}}),1\}$, which places $T_e$ ten times below $f_{\mathrm{med}}$ so that the active arcs of the start begin in the logarithmic tail. Dense flow starts are instead over-smoothed with $\lambda_0=10^{2}\max\{\max_e|f^{(0)}_e|,1\}$. We set $\mu_0=\min\{\theta_0\lambda_0,\ \min_{e\in E}\mu_{\max,e}(\lambda_0)\}$ with $\theta_0=0.3$ and $0.7$, respectively. Both initial parameter pairs are admissible under \Cref{prop:fixed-parameter-majorization} (\Cref{app:fixed-parameter-majorization}) by construction.

The update candidate for $\lambda_{k+1}$ is the gap-based value $\lambdagap=f_{\downarrow}^{\,2}/(2b_{\mathrm{rep}})$, where $b_{\mathrm{rep}}$ is a representative fixed cost of the arcs near the gap defined by $K_{\text{gap}}$ and is motivated by the scale match of $
T_e\approx\sqrt{2\lambdagap b_{\mathrm{rep}}}=f_\downarrow$ for negligible $r_e$. Within a stage, we set $\lambda_{k+1}=\max\{\min\{\lambda_{k},\lambdagap\},\lambda_{k}/5\}$, whereas we set $\lambda_{k+1}=\max\{\min\{\lambda_{k}/4,\lambdagap\},\lambda_{k}/100,10^{-9}\}$ across stage boundaries. 
We use this acceleration only after the detected gap is
sufficiently pronounced and stable; the trust conditions and remaining
schedule details are given in \Cref{app:smoothing_details}.

For exact weights and an exact QP solve, \Cref{prop:fixed-parameter-majorization} (\Cref{app:fixed-parameter-majorization}) gives
$J_{\mathrm{LL}}^{\lambda_k,\mu_k}(f^{(k+1)})
\leq J_{\mathrm{LL}}^{\lambda_k,\mu_k}(f^{(k)})$;
acceptance/backtracking, with rejection as a fallback, enforces the same
current-pair nonincrease numerically. Continuation changes the surrogate, so
no monotonic comparison across parameter pairs is claimed. Thus,
fixed-pair descent supplies stability while continuation supplies sharpening.
The best feasible incumbent remains nonincreasing under $J_{\mathrm{true}}$.
\section{Algorithmic Layers}
\label{sec:algo_layers}

The IRLS procedure provides a continuous mechanism for refining a feasible flow under the smoothed fixed-charge surrogate. However, because the surrogate is nonconvex, the final support can strongly depend on the starting point. We therefore use a layered design that combines continuous reweighting with structured support exploration. Each layer maintains the best feasible solution, scored by the true fixed-charge objective $J_{\mathrm{true}}$.

\paragraph{Layer 1: Multi-start IRLS.}
The first layer runs the IRLS core method of \Cref{alg:irls_core} from five complementary feasible initializations: a linear-cost-weighted quadratic-flow solution, a quadratic-flow solution weighted by the combined linear and capacity-amortized fixed cost, a linearized MCF solution, a capacity-amortized MCF solution, and a DSSP slope-scaling iterate; implementation details are given in supplementary material, \Cref{app:algo_layers}. Duplicate supports are removed, and IRLS is run from each remaining start. This layer gives the algorithm access to several plausible support regions without introducing binary activation variables.

\paragraph{Layer 2: Objective-driven perturbation restarts.}
The second layer explores alternatives around each support discovered by Layer 1. Given an incumbent flow $f$, active arcs are scored by their realized average cost $r_e+b_e/\max\{f_e,1\}$, while inactive arcs are scored by their capacity-amortized opening cost $r_e+b_e/\max\{c_e,1\}$. These scores identify expensive active arcs that may be worth suppressing and cheap inactive arcs worth opening. Using these scores, we form a perturbed target $\tilde f$ by cutting flow on the expensive active arcs and putting flow on the cheap inactive ones.
The resulting perturbed targets are not required to be feasible; instead, each target is projected back onto the capacitated flow polytope $\mathcal{F} = \{g\in\mathbb{R}^{|E|}: Cg=d, \; 0\leq g\leq c\}$ via $\Pi_{\mathcal F}(\tilde f) = \arg\min_{g\in\mathcal F}\|g-\tilde f\|_2^2$.
Projections that fail or duplicate a previously generated support are discarded. The surviving projections are ranked using their true objective values, support changes, and numbers of active arcs, and the highest-ranked candidates are used as new IRLS starts. This layer acts as a structured escape mechanism from locally stable but suboptimal supports. Details of the perturbation rules, candidate selection, and projection step are given in supplementary material, \Cref{app:algo_layers}.

\paragraph{Layer 3: Anchor-union restricted IRLS.}
The third layer expands the search beyond the supports discovered by IRLS and its perturbation restarts. This is important because some useful arcs may never enter an IRLS trajectory once the surrogate becomes sharp, especially when opening them requires coordinated changes across an entire path. We therefore augment the IRLS supports with supports obtained from external FCNFP algorithms. Specifically, we use ADSSP solutions generated from the \texttt{max\_flow}, \texttt{unitflow}, \texttt{trans}, and \texttt{multi\_unit} initializations as support anchors. These solutions often provide sparse feasible supports containing useful combinatorial structure, even when they do not attain the best objective value.

Let $\mathcal X_{\mathrm{IRLS}}$ denote the candidate flows from Layer 2, and let $\mathcal X_{\mathrm{anc}}$ denote anchor flows from the external solver. For any flow $x$, define its support $S(x)=\{e\in E: x_e>\tau\}$, where $\tau$ is the arc activation threshold. Each candidate union support has the form $E' = S(f)\cup S(a)$, where $f\in\mathcal X_{\mathrm{IRLS}}$ and $S(a)$ is empty, a single anchor support, or the union of all anchor supports. Duplicate, empty, and infeasible supports are discarded.

For each retained union $E'$, we compute feasible refit flows by solving restricted linear flow problems with costs consisting of variable cost and capacity-amortized fixed cost. Several fixed-cost weights are tried, together with the corresponding Layer 2 flow, and the best feasible starts are retained. We then run IRLS on the subproblem obtained by fixing all arcs outside $E'$ to zero. A lightweight within-union refinement then tests deleting expensive active arcs and reintroducing promising inactive ones, keeping only feasible improvements. All candidates are scored under the original fixed-charge objective $J_{\mathrm{true}}$, and the best feasible solution is retained. Further implementation details are given in supplementary material, \Cref{app:algo_layers}. \Cref{alg:anchor_union_irls} gives this complete anchor-union restricted IRLS procedure.

\begin{algorithm}[t]
\caption{Layer 3: Anchor-Union Restricted IRLS}
\label{alg:anchor_union_irls}
\begin{algorithmic}[1]
\Require Network data $(C,d,c,r,b)$, IRLS starts $\mathcal I$, ADSSP starts $\mathcal{A}$
\State $\mathcal X \gets \{\textsc{IRLS}(f^{(0)}): f^{(0)}\in\mathcal I\}$ \Comment{Layer 1}
\State $\mathcal P \gets \{\Pi_{\mathcal F}(\tilde f):\tilde f\in\textsc{Perturb}(\mathcal X)\}$
\State $\mathcal X_{\mathrm{IRLS}} \gets \mathcal X \cup\{\textsc{IRLS}(p):p\in\mathcal P,\ p\text{ feasible}\}$
\State $x^\star \gets \arg\min_{f\in\mathcal X_{\mathrm{IRLS}}} J_{\mathrm{true}}(f)$;\; $J^\star \gets J_{\mathrm{true}}(x^\star)$ \Comment{Layer 2}
\State $\mathcal X_{\mathrm{anc}} \gets \{\textsc{ADSSP}(i): i\in\mathcal A\}$
\State $\mathcal A^+ \gets \{\varnothing\}\cup\{S(a):a\in\mathcal X_{\mathrm{anc}}\}\cup\{\textstyle\bigcup_{a\in\mathcal X_{\mathrm{anc}}} S(a)\}$
\State $\mathcal U \gets \{\,S(f)\cup A : f\in\mathcal X_{\mathrm{IRLS}},\ A\in\mathcal A^+\,\}$
\For{$E' \in \mathcal U$}
    \State $\mathcal C(E') \gets \{f|_{E'} : f\in\mathcal X_{\mathrm{IRLS}}\} \cup \{\textsc{Refit}(E')\}$
    \State $\bar x \gets \textsc{IRLS}(E';\,\mathcal C(E'))$
    \State $\hat x \gets \textsc{Refine}(\bar x, E')$ \Comment{local add/drop within $E'$}
    \If{$J_{\mathrm{true}}(\hat x) < J^\star$} $(x^\star,J^\star)\gets(\hat x,\,J_{\mathrm{true}}(\hat x))$ \EndIf
\EndFor
\State \Return $x^\star$
\end{algorithmic}
\end{algorithm}

\section{Experiments}
\label{sec:experiments}
We evaluate the proposed method on a collection of $410$ single-commodity FCNFP instances drawn from benchmark, synthetic, and large-scale families. Benchmark sets support comparison to prior work; scale families probe MILP-intractable sizes. The set includes \texttt{FCNetLib\_fc} from \cite{atamturk2001flow}, \texttt{benchmark} from \cite{zibBenchmark}, synthetically generated \texttt{synthetic} instances, and nine \texttt{scale} families from \cite{yang2024sequential}. 
All methods are evaluated using the true fixed-charge objective in \eqref{eq:jtrue}, with arc
activation threshold $\tau=10^{-6}$. A returned flow is treated as feasible only if the flow-conservation residual, capacity violation, and nonnegativity violation are each at most $10^{-6}$.

We compare the proposed method against three FCNFP baselines and a MILP reference. The baselines are linearized MCF, DSSP, and ADSSP. Linearized MCF solves a min-cost-flow problem using arc costs that approximate the combined fixed and proportional costs. DSSP \cite{kim1999solution} and ADSSP \cite{yang2024sequential} are described in \Cref{sec:related_work}. When the MILP proves optimality, its objective provides a certified optimum; otherwise, its feasible incumbent is used only as a time-limited reference.

In our experiments, L1, L2 and L3 denote Layers 1--3 as described in \Cref{sec:algo_layers}. All heuristic methods use multiple feasible starts when applicable, and the best feasible solution under $J_{\mathrm{true}}$ is reported. Each method is given a solver time limit of $300$ seconds per instance. Reported times are wall-clock and include model construction and solution extraction, so they can exceed this limit. Each individual IRLS run is limited to $200$ outer reweighting iterations.

We report two objective-gap metrics. For a method $z$ on instance $i$ with objective $J_z(i)$, the gap relative to MILP is
\[
\mathrm{Gap}_{\mathrm{MILP}}(z,i) = 100\cdot \frac{J_z(i)- J_{\mathrm{MILP}}(i)}{|J_{\mathrm{MILP}}(i)|},
\]
where $J_{\mathrm{MILP}}(i)$ is either a certified optimum or a time-limited MILP incumbent. 
The gap relative to the best non-MILP method is
\[
\mathrm{Gap}_{\mathrm{best}}(z,i) = 100\cdot \frac{J_z(i)- J_{\mathrm{best}}(i)}{|J_{\mathrm{best}}(i)|},
\]
where $J_{\mathrm{best}}(i)$ is the best feasible objective among the non-MILP methods on instance $i$.

Full hardware, software versions, random seeds, and instance-generation details are given in supplementary material, \Cref{app:experiment_setup}.

\paragraph{Results.}
\Cref{tab:overall} reports the aggregate comparison over all $410$ instances with gaps taken relative to the MILP objective---a certified optimum where the MILP proved optimality, and a time-limited incumbent otherwise. MILP proved optimality on $235$ of the $410$ instances. L3 achieves the best objective quality among all heuristics, with a mean gap of $1.316\%$ and a median gap of $0.131\%$. This improves over L2, L1, ADSSP, DSSP, and MCF. Here, L3$^-$ denotes L3 with anchors disabled; it still outperforms ADSSP, showing that L3's gain does not depend on its anchor supports alone (see Supplementary Material \Cref{app:results}).

\begin{table}[t]
\centering
\caption{Aggregate performance on the $410$ instances with a matched feasible MILP baseline.  Gaps are relative to the MILP objective; times are seconds.}
\label{tab:overall}
\small
\setlength{\tabcolsep}{4pt}
\begin{tabular}{lrrrrrr}
\toprule
& \multicolumn{4}{c}{Gap to MILP baseline (\%)} &
\multicolumn{2}{c}{Time (s)} \\
\cmidrule(lr){2-5}\cmidrule(lr){6-7}
Method & mean & p25 & median & p75 & mean & median \\
\midrule
MILP & 0.000 & 0.000 & 0.000 & 0.000 & 144.939 & 64.516 \\
\midrule
L3    & \textbf{1.316} & \textbf{0.000} & \textbf{0.131} & \textbf{1.278} & 17.943 & 8.074 \\
L3$^-$  & 2.073 & 0.013 & 0.560 & 2.060 & 6.523 & 1.775 \\
L2    & 2.334 & 0.037 & 0.749 & 2.424 & 5.999 & 2.632 \\
L1    & 3.420 & 0.077 & 1.412 & 3.893 & 1.707 & 0.723 \\
ADSSP & 2.931 & 0.027 & 0.695 & 3.792 & 0.370 & 0.151 \\
DSSP  & 4.084 & 0.045 & 2.164 & 5.903 & 1.391 & 0.369 \\
MCF   & 8.304 & 0.118 & 4.199 & 11.577 & \textbf{0.089} & \textbf{0.043} \\
\bottomrule
\end{tabular}
\end{table}

\Cref{tab:heuristic_ablation} reports the non-MILP methods only comparison on all $410$ instances. This table isolates the contribution of the algorithmic layers without relying on a MILP reference. The base IRLS layer has a mean gap of $2.163\%$ to the best displayed heuristic. L2 reduces this gap to $1.106\%$, while L3 reduces it further to \(0.140\%\). L3 is also the best or tied-best method on $90.0\%$ of instances.

\begin{table}[t]
\centering
\caption{Heuristic layer-ablation performance over all $410$ instances.
Gaps are relative to the best displayed heuristic on each instance.}
\label{tab:heuristic_ablation}
\small
\setlength{\tabcolsep}{6pt}
\begin{tabular}{lrrr}
\toprule
Method & Mean gap (\%) & Median gap (\%) & Win/tie rate (\%) \\
\midrule
L3    & \textbf{0.140} & \textbf{0.000} & \textbf{90.0} \\
L2    & 1.106 & 0.441 & 22.9 \\
L1    & 2.163 & 1.100 & 12.0 \\
ADSSP & 1.709 & 0.402 & 16.6 \\
DSSP  & 2.848 & 1.468 & 12.2 \\
MCF   & 6.903 & 3.765 & 6.1 \\
\bottomrule
\end{tabular}
\end{table}

Family-wise and MILP-status breakdowns further show that L3's advantage persists across instance scales and MILP certification status (see Supplementary Material \Cref{app:results}).

\paragraph{Discussion.}
The experiments indicate that support quality is the primary determinant of performance in large-scale FCNFP. IRLS (\Cref{alg:irls_core}) provides a suitable continuous mechanism for identifying flows with locally minimal fixed-charge cost objectives, with its final support still related to the initial flow vector due to the inherent combinatorial problem structure. Perturbation restarts (introduced via L2) empirically improve final flow costs by exploring nearby support with the anchor-union layer of the L3 pipeline of \Cref{alg:anchor_union_irls} leading to further performance increase to achieve the state-of-the-art among non-MILP methods. The intuition is that a strictly \emph{better} support can be found by jointly considering IRLS-discovered supports with support candidates identified from other methodologies, before reoptimizing via IRLS continuously over the resulting restricted subgraphs. The ADSSP anchor supports, which are part of the L3 pipeline, help to improve performance over L1 or L2, with L3, however, improving performance beyond a pure ADSSP method through its IRLS-enabled support discovery. 

The runtime results show the expected quality--time tradeoff. MCF and ADSSP are attractive when solutions are needed almost immediately. L3 is more expensive because it performs multiple IRLS solves, restricted refits, and support refinements. However, it remains far cheaper than MILP on average while consistently providing the strongest heuristic objective values.
Therefore, the method is best viewed as a high-quality primal algorithm for settings where a strong feasible solution is more important than a certificate of global optimality.

\section{Conclusion}
\label{sec:conclusion}
We proposed a scalable continuous-optimization algorithm based on IRLS for the fixed-charge network flow problem. The method replaces the discontinuous fixed-plus-linear cost objective with a smooth Lasry--Lions surrogate, solves the resulting weighted quadratic flow subproblems using Laplacian-structured linear algebra, and augments the continuous reweighting process with perturbation restarts and anchor-union restricted search.

Across $410$ instances, our method achieves the best objective quality among the evaluated FCNFP baselines, including on large-scale instances where the MILP cannot certify optimality. These results show that combining smooth continuous optimization with structured support exploration is an effective strategy for large-scale FCNFP.

\section*{Acknowledgments}
This work was supported in part by the NSF grant CCF-2549926.

\bibliographystyle{unsrt}
\bibliography{reference}
\appendix
\section{Supplementary Material}
\subsection{Derivation of the Lasry--Lions Surrogate}
\label{app:ll_derivation}
This section derives the Lasry--Lions double envelope surrogate used in the main paper (\Cref{sec:lasrylions}): the first Moreau envelope, the exact double envelope, the logarithmic active-flow modification that yields the implemented surrogate, the admissible range for $\mu$, continuous differentiability of the resulting surrogate, and the proof that it converges pointwise to the true fixed-charge objective as $\lambda\downarrow0$.

Throughout this subsection we fix a single, arbitrary edge and suppress its
index: we write $r$ and $b$ for its proportional and fixed cost, and likewise
$J$, $T$, $S$, $x_1$, $x_2$, $\psi_2$, $\mu_{\max}$ and
$\mathrm{LL}^{\lambda,\mu}$ for the corresponding per-edge quantities. We assume the nonnegative-cost setting
$r_e\geq0$, $b_e\geq0$ and $r_e+b_e>0$ for every edge $e\in E$; for the fixed
edge this reads $r,b\geq0$ and $r+b>0$. The per-edge objective is
\[
J(f) = r|f|+b\,\mathbf{1}\{|f|>0\}.
\]
For the directed FCNFP, $f\geq 0$, but the derivation is written in terms of $t=|f|$ so that it also applies to a signed-flow formulation.

\paragraph{First Moreau envelope.}
For $\lambda>0$, define the Moreau envelope \cite{moreau1965proximite}
\[
M^\lambda J(t) = \inf_{u\geq 0} \left\{ J(u)+\frac{(t-u)^2}{2\lambda}\right\}.
\]
There are two relevant candidates. Setting $u=0$ gives
\[
q_0(t)=\frac{t^2}{2\lambda}.
\]
For $u>0$, the objective is
\[
b+ru+\frac{(t-u)^2}{2\lambda}.
\]
Its stationary point is $u=t-\lambda r$, which gives
\[
q_1(t) = b+rt-\frac{\lambda r^2}{2}.
\]
The two candidates coincide when
\[
\frac{t^2}{2\lambda} = b+rt-\frac{\lambda r^2}{2},
\]
or equivalently,
\[
(t-\lambda r)^2=2\lambda b.
\]
Hence, the transition point is
\[
T=\lambda r+\sqrt{2\lambda b},
\]
and the first Moreau envelope is
\begin{equation}
\label{eq:first_moreau}
M^\lambda J(t) =
\begin{cases}
\dfrac{t^2}{2\lambda},
& 0\leq t\leq T,\\[8pt]
b+rt-\dfrac{\lambda r^2}{2},
& t > T.
\end{cases}
\end{equation}

\paragraph{Lasry--Lions double envelope.}
For $0<\mu<\lambda$, the Lasry--Lions double envelope can be written as
\[
L^{\lambda,\mu} J(t) = \sup_{u\geq 0} \left\{ M^\lambda J(u)-\frac{(t-u)^2}{2\mu} \right\}.
\]

Consider first the quadratic branch of \eqref{eq:first_moreau}. The maximization
problem becomes
\[
\sup_{0\leq u\leq T} \left\{ \frac{u^2}{2\lambda}-\frac{(t-u)^2}{2\mu}\right\}.
\]
The unconstrained maximizer is
\[
u^\star=\frac{\lambda t}{\lambda-\mu}.
\]
It remains in the quadratic branch when $u^\star\leq T$, or
\[
t\leq x_1, \qquad x_1=\frac{\lambda-\mu}{\lambda}T.
\]
Substitution gives
\[
L^{\lambda,\mu} J(t) = \frac{t^2}{2(\lambda-\mu)}, \qquad 0\leq t\leq x_1.
\]

When $t>x_1$, the quadratic-branch maximizer would lie beyond $T$.
For an intermediate range, the maximum is therefore attained at the junction
$u=T$, giving
\[
L^{\lambda,\mu} J(t) = \frac{T^2}{2\lambda}-\frac{(t-T)^2}{2\mu}.
\]
The derivative of this middle branch is
\[
\frac{T-t}{\mu}.
\]

For the linear branch of \eqref{eq:first_moreau}, the maximizing point is $u^\star=t+\mu r$. It lies in the linear branch when $t+\mu r\geq T$, or equivalently  $t\geq S:=T-\mu r$.

Substitution gives the exact Lasry--Lions tail
\[
L^{\lambda,\mu}J(t) = b+rt-\frac{(\lambda-\mu)r^2}{2},
\qquad t\geq S.
\]
Thus, the exact Lasry--Lions double envelope is
\begin{equation}
\label{eq:exact_ll}
L^{\lambda,\mu} J(t) =
\begin{cases}
\dfrac{t^2}{2(\lambda-\mu)},
& 0\leq t\leq x_1,\\[9pt]
\dfrac{T^2}{2\lambda} -\dfrac{(t-T)^2}{2\mu},
& x_1<t<S,\\[9pt]
b+rt-\dfrac{(\lambda-\mu)r^2}{2},
& t\geq S,
\end{cases}
\end{equation}
which is continuously differentiable at both transition points.

\paragraph{Logarithmic active-flow modification.}
The exact Lasry--Lions tail in \eqref{eq:exact_ll} has derivative $r$.
Consequently, an IRLS weight based on $L'(t)/t$ would equal $r/t$ in the active-flow region.
The fixed charge $b$ would appear only as an additive constant and would therefore disappear from the reweighting rule.

To retain a fixed-charge contribution in the active-flow weights, we replace the exact linear tail by the logarithmic continuation
\begin{equation}
\label{eq:log_tail_ansatz}
\psi_3(t) = \psi_2(x_2) + r(t-x_2) +\alpha b\log\left(\frac{t}{x_2}\right), \qquad t\geq x_2,
\end{equation}
where
\[
\psi_2(t) = \frac{T^2}{2\lambda} - \frac{(t-T)^2}{2\mu}
\]
is the middle Lasry--Lions branch. The coefficient
\[
\alpha = \frac{\lambda}{\lambda+s\lambda_0},
\qquad
s=10^{-2},
\]
is shared across all edges and controls the strength of the logarithmic term.

Value continuity at $x_2$ holds automatically because of the additive constant $\psi_2(x_2)$. To obtain derivative continuity, we require
\[
\psi_2'(x_2)=\psi_3'(x_2).
\]
Since
\[
\psi_2'(t)=\frac{T-t}{\mu}
\qquad\text{and}\qquad
\psi_3'(t)=r+\frac{\alpha b}{t},
\]
the matching condition is
\[
\frac{T-x_2}{\mu} = r+\frac{\alpha b}{x_2}.
\]
Multiplication by $\mu x_2$ gives
\[
x_2(T-x_2) = \mu r x_2+\mu\alpha b,
\]
and therefore
\[
x_2^2-(T-\mu r)x_2+\mu\alpha b=0.
\]
Defining
\[
S=T-\mu r,
\]
the two roots are
\[
x_2^\pm = \frac{S\pm\sqrt{S^2-4\mu\alpha b}}{2}.
\]
We select the larger root,
\begin{equation}
\label{eq:x2_derivation}
x_2 = \frac{S+\sqrt{S^2-4\mu\alpha b}}{2},
\end{equation}
because $x_2\to S$ as $\alpha\to 0$, thereby recovering the transition point of the exact Lasry--Lions envelope. The smaller root instead converges to zero.

The resulting implemented surrogate is
\begin{equation}
\label{eq:combined_ll_derivation}
\mathrm{LL}^{\lambda,\mu}(x) =
\begin{cases}
\dfrac{|x|^2}{2(\lambda-\mu)}, & |x|\leq x_1,\\[10pt]
\dfrac{T^2}{2\lambda} -\dfrac{(|x|-T)^2}{2\mu}, & x_1<|x|<x_2,\\[10pt]
\psi_2(x_2) +r(|x|-x_2) + \\
\alpha b\log\left(\dfrac{|x|}{x_2}\right), & |x|\geq x_2.
\end{cases}
\end{equation}
It preserves the two Lasry--Lions parabolic regimes near the origin while introducing a logarithmic active-flow tail that retains the fixed-charge information required by IRLS.

\paragraph{Admissible values of $\mu$.}\label{appendix:mu_bound}
For $x_2$ to be real, its discriminant must be nonnegative:
\[
(T-\mu r)^2-4\mu\alpha b\geq 0.
\]
Expanding gives
\[
r^2\mu^2-(2Tr+4\alpha b)\mu+T^2\geq 0.
\]
Let
\[
A=2Tr+4\alpha b.
\]
For $r>0$, the admissible small-$\mu$ branch is
\[
\mu\leq\frac{A-\sqrt{A^2-4r^2T^2}}{2r^2}.
\]
Rationalizing this expression gives the numerically stable bound
\begin{equation}
\label{eq:mu_bound_derivation}
\mu_{\max} = \frac{2T^2}{A+\sqrt{A^2-4r^2T^2}}.
\end{equation}
Therefore,
\begin{equation}
\label{eq:mumax-appendix}
\mu_{\max}=\frac{2T^2}{2T r+4\alpha b+\sqrt{(2T r+4\alpha b)^2-4r^2T^2}}.
\end{equation}
In the implementation, $\mu$ is capped so that $\mu<\lambda$ and $\mu\le\min_{e\in E}\mu_{\max,e}$, where $\mu_{\max,e}$ denotes \eqref{eq:mumax-appendix} evaluated at the costs of edge $e$, so that the bound holds simultaneously on every edge. When $r$ is zero, this bound reduces to $\lambda/(2\alpha)$.

\paragraph{Continuous Differentiability of the Lasry--Lions Surrogate.}
We verify that $\mathrm{LL}^{\lambda,\mu}$ in \eqref{eq:combined_ll_derivation} is continuously differentiable on $\mathbb R$ by differentiating each branch directly and checking that values and derivatives match at $x_1$ and $x_2$. Since every branch depends on $x$ only through $|x|$, it suffices to check $x\ge0$; the negative side follows by evenness. At $x=0$, the quadratic branch $x^2/(2(\lambda-\mu))$ is the same expression on both sides of the origin, with derivative $x/(\lambda-\mu)\to0$ as $x\to0^\pm$, so differentiability at $0$ is immediate.

\emph{At $x_1$.} Differentiating the quadratic and middle branches gives
\begin{align*}
\frac{d}{dx}\frac{x^2}{2(\lambda-\mu)} &= \frac{x}{\lambda-\mu},\\
\frac{d}{dx}\left[\frac{T^2}{2\lambda}-\frac{(x-T)^2}{2\mu}\right] &= \frac{T-x}{\mu}.
\end{align*}
Using $x_1=(\lambda-\mu)T/\lambda$,
\begin{align*}
\frac{x_1}{\lambda-\mu} &= \frac{(\lambda-\mu)T/\lambda}{\lambda-\mu} = \frac{T}{\lambda},\\
\frac{T-x_1}{\mu} &= \frac{T-(\lambda-\mu)T/\lambda}{\mu}\\
&= \frac{T\mu/\lambda}{\mu} = \frac{T}{\lambda},
\end{align*}
so both derivatives equal $T/\lambda$ at $x_1$. The values match as well: since $x_1-T=T\big[(\lambda-\mu)-\lambda\big]/\lambda=-\mu T/\lambda$,
\begin{align*}
\frac{T^2}{2\lambda}-\frac{(x_1-T)^2}{2\mu}
&=\frac{T^2}{2\lambda}-\frac{\mu T^2}{2\lambda^2}\\
&=\frac{T^2}{2\lambda}\cdot\frac{\lambda-\mu}{\lambda}\\
&=\frac{(\lambda-\mu)T^2}{2\lambda^2}
=\frac{x_1^2}{2(\lambda-\mu)},
\end{align*}
the last step using $x_1^2=(\lambda-\mu)^2T^2/\lambda^2$. This matches the quadratic branch evaluated at $x_1$, so the two branches meet with equal value and equal slope.

\emph{At $x_2$.} The middle branch has derivative $(T-x)/\mu$ as above; differentiating the log-tail branch $\psi_2(x_2)+r(x-x_2)+\alpha b\log(x/x_2)$ gives
\[
\frac{d}{dx}\Big[r(x-x_2)+\alpha b\log(x/x_2)\Big] = r+\frac{\alpha b}{x}.
\]
Equality of the two derivatives at $x=x_2$, namely $(T-x_2)/\mu=r+\alpha b/x_2$, holds if and only if, after multiplying both sides by $\mu x_2$,
\begin{align*}
x_2(T-\mu r)-x_2^2-\mu\alpha b&=0\\
\Longleftrightarrow\quad x_2^2-Sx_2+\mu\alpha b&=0,
\end{align*}
using $S=T-\mu r$. This is exactly the defining quadratic solved for $x_2$ in \eqref{eq:x2_derivation}, so the derivative equality holds exactly, not merely in the limit. Value continuity at $x_2$ is immediate: the log-tail branch evaluates to $\psi_2(x_2)+r(x_2-x_2)+\alpha b\log(x_2/x_2)=\psi_2(x_2)$, matching the middle branch value at $x_2$ by definition of $\psi_2(x_2)$.

\emph{Degenerate case $x_1=x_2$.} When $b=0$, the middle interval collapses to a point, $x_1=x_2=(\lambda-\mu)r$ (\Cref{app:fixed-parameter-majorization}). Neither derivative computation above assumed $x_1<x_2$, so both identities still hold verbatim at this shared point: the quadratic, middle, and log-tail derivatives all equal $T/\lambda=r$ there, and the quadratic and log-tail branches meet directly with equal value and slope even though the middle branch has zero width.

Hence $\mathrm{LL}^{\lambda,\mu}$ is continuously differentiable everywhere, with
\[
\big(\mathrm{LL}^{\lambda,\mu}\big)'(x) = \operatorname{sign}(x)\,w(x)\,|x|, \qquad x\neq0,
\]
where $w(x)=(\mathrm{LL}^{\lambda,\mu})'(|x|)/|x|$ is the weight function of \eqref{eq:weight_ll} of the main paper, taken at $x=0$ by continuous extension.

\paragraph{Proof of Pointwise Convergence.}
Fix a path $\mu(\lambda)$ such
that, for every $\lambda>0$,
\begin{equation*}
0<\mu(\lambda)<\lambda,
\qquad
\mu(\lambda)\leq
\min_{e\in E}\mu_{\max,e}(\lambda),
\end{equation*}
where $\mu_{\max,e}(\lambda)$ denotes \eqref{eq:mumax-appendix} evaluated at
the costs of edge $e$. This is the admissible class
\eqref{eq:general-admissible-mu}; it is nonempty, since $T_e>0$ implies
$\mu_{\max,e}(\lambda)>0$, and it covers the implemented rule
$\mu(\lambda)=\min\{\lambda/2,\ \min_{e\in E}\mu_{\max,e}(\lambda)\}$ of
\Cref{sec:lambda_update} as a special case. Since $0<\mu(\lambda)<\lambda$, it follows that $\mu(\lambda)\to0$ as
$\lambda\downarrow0$. Let further
\[
\alpha(\lambda)=\frac{\lambda}{\lambda+s\lambda_0},
\]
where $s>0$ and $\lambda_0>0$ are fixed as in \Cref{sec:lasrylions}. Recall the standing assumptions $r,b\geq0$ and $r+b>0$. We show that
$\mathrm{LL}^{\lambda,\mu(\lambda)}(x)\to
J(x)=r|x|+b\,\mathbf 1\{|x|>0\}$ for every fixed $x\in\mathbb R$ as
$\lambda\downarrow0$. Summing over edges then gives
$J_{\mathrm{LL}}^{\lambda,\mu(\lambda)}(f)\to J_{\mathrm{true}}(f)$
for every fixed $f\in\mathcal F$.

\emph{Case $x=0$.} The origin always lies in the quadratic branch, and $\mathrm{LL}^{\lambda,\mu}(0)=0=J(0)$ for every $\lambda$, so the limit is immediate.

\emph{Transition points vanish.} Since $r,b\geq0$,
\begin{align*}
T(\lambda)&=\lambda r+\sqrt{2\lambda b}\longrightarrow0,\\
S(\lambda)&=(\lambda-\mu(\lambda))r+\sqrt{2\lambda b}.
\end{align*}
Since $\mu(\lambda) r \geq 0$ and $\lambda - \mu(\lambda) > 0$ together with $r+ b > 0$, we have
$0<S(\lambda)\leq T(\lambda)$. Hence $S(\lambda)\to0$ by the squeeze
theorem.
The admissibility condition guarantees that
$S^2-4\mu\alpha b\geq0$. Since $\mu\alpha b\geq0$, we also have
\[
\begin{aligned}
0&\leq S^2-4\mu\alpha b\leq S^2,\\
0&\leq\sqrt{S^2-4\mu\alpha b}\leq S.
\end{aligned}
\]
It follows directly from the definition of the larger root that
\[
\frac{S}{2}
\leq
x_2=\frac{S+\sqrt{S^2-4\mu\alpha b}}{2}
\leq S.
\]
Therefore $x_2(\lambda)\to0$ by the squeeze theorem. Finally, $0\leq x_1\leq x_2$, so $x_1(\lambda)\to0$ as well.

\emph{Case $x\neq0$, $b>0$.} Because $x_2(\lambda)\to0$, there is $\lambda^\ast>0$ with $x_2(\lambda)<|x|$ for all $0<\lambda<\lambda^\ast$, so $x$ lies in the log-tail branch throughout this range:
\[
\mathrm{LL}^{\lambda,\mu}(x)=\psi_2(x_2)+r(|x|-x_2)+\alpha b\log(|x|/x_2).
\]
We take each term to its limit.

The linear term satisfies $r(|x|-x_2)\to r|x|$ since $x_2\to0$.

For the logarithmic term, $S\geq\sqrt{2\lambda b}$ and
$x_2\geq S/2$ give $x_2\geq\sqrt{\lambda b/2}$. Moreover,
$0<\alpha\leq\lambda/(s\lambda_0)$. Hence, for all sufficiently small
$\lambda$,
\begin{align*}
0
&\leq \alpha b\log\!\left(\frac{|x|}{x_2}\right)\\
&\leq
\frac{\lambda b}{s\lambda_0}
\log\!\left(|x|\sqrt{\frac{2}{\lambda b}}\right)\\
&=
\frac{b}{s\lambda_0}
\left[
\lambda\log\!\left(|x|\sqrt{\frac{2}{b}}\right)
+\frac{\lambda}{2}\log\!\left(\frac{1}{\lambda}\right)
\right]\longrightarrow0,
\end{align*}
where the last limit uses $\lambda\log(1/\lambda)\to0$. The squeeze
theorem therefore gives $\alpha b\log(|x|/x_2)\to0$.

It remains to determine the limit of the junction value $\psi_2(x_2)$.
The defining relation $x_2^2-Sx_2+\mu\alpha b=0$ gives
$x_2(S-x_2)=\mu\alpha b$. Since $x_2>0$,
\begin{align*}
T-x_2&=(T-S)+(S-x_2)\\
&=\mu r+\frac{\mu\alpha b}{x_2}\\
&=\mu\left(r+\frac{\alpha b}{x_2}\right).
\end{align*}
The preceding lower bound on $x_2$ yields the explicit estimate
\[
0\leq\frac{\alpha b}{x_2}
\leq
\frac{2\alpha b}{\sqrt{2\lambda b}}
=\frac{\sqrt{2b\lambda}}{\lambda+s\lambda_0}
\leq\frac{\sqrt{2b\lambda}}{s\lambda_0}
\longrightarrow0.
\]
Thus $r+\alpha b/x_2\to r$ and, since $\mu\to0$,
\[
\frac{(x_2-T)^2}{2\mu}
=\frac{\mu}{2}\left(r+\frac{\alpha b}{x_2}\right)^2
\longrightarrow0.
\]
On the other hand, the exact identity
\[
\frac{T^2}{2\lambda}
=\frac{\lambda r^2}{2}+r\sqrt{2\lambda b}+b
\]
shows that $T^2/(2\lambda)\to b$. Hence
\[
\psi_2(x_2)=\frac{T^2}{2\lambda}-\frac{(x_2-T)^2}{2\mu}\ \longrightarrow\ b.
\]
Adding the three limits gives $\mathrm{LL}^{\lambda,\mu(\lambda)}(x)\to b+r|x|=J(x)$.

\emph{Case $x\neq0$, $b=0$.} Then $r>0$, and $x_1=x_2=(\lambda-\mu)r\to0$ since $0<\lambda-\mu<\lambda\to0$, so $x$ again eventually lies in the log-tail branch, where $\alpha b=0$ exactly. Here $S=x_2$ and $T=S+\mu r$, so
\begin{align*}
\psi_2(x_2)&=\frac{T^2}{2\lambda}-\frac{(\mu r)^2}{2\mu}\\
&=\frac{(\lambda-\mu)r^2}{2}\to0=b,
\end{align*}
and $r(|x|-x_2)\to r|x|$, so again $\mathrm{LL}^{\lambda,\mu(\lambda)}(x)\to r|x|=J(x)$.

In every case $\mathrm{LL}^{\lambda,\mu(\lambda)}(x)\to J(x)$ as $\lambda\downarrow0$, which proves the claim.

\subsection{Quadratic Majorization at Fixed Smoothing Parameters}
\label{app:fixed-parameter-majorization}

We next justify the majorization--minimization interpretation of the quadratic IRLS model. Throughout this subsection, the smoothing parameters $(\lambda,\mu)$ are fixed. We assume the standard nonnegative-cost setting $r_e,b_e\geq0$ and $r_e+b_e>0$ for every $e\in E$, and note that the implemented coefficient satisfies $0<\alpha\leq1$. On instance families where some arcs
have $r_e=b_e=0$ exactly, this assumption is violated as stated; the implementation floors $r_e$ at $10^{-12}$ before evaluating the surrogate, its weights, and $\mu_{\max,e}$ from \eqref{eq:mumax-appendix}, which restores $r_e+b_e>0$ for every edge and brings such arcs within the
scope of the proposition below. We take
\begin{equation}
\label{eq:general-admissible-mu}
0<\mu<\lambda,
\qquad
\mu\leq\min_{e\in E}\mu_{\max,e}(\lambda).
\end{equation}
The $\mu_{\max,e}$ cap makes the discriminant defining $x_{2,e}$ nonnegative.
To show that the transition points in \Cref{sec:lasrylions} of the main paper are also ordered,
set
\[
S_e=T_e-\mu r_e
=(\lambda-\mu)r_e+\sqrt{2\lambda b_e}>0
\]
and
\[
p_e(x)=x^2-(T_e-\mu r_e)x+\mu\alpha b_e.
\]
Since the discriminant is nonnegative, the root sum is $S_e>0$, and the root product is $\mu\alpha b_e\ge0$, both roots are real and nonnegative, and $x_{2,e}$ is the larger one.
If $b_e=0$, then $r_e>0$ and
$p_e(x)=x\bigl(x-(\lambda-\mu)r_e\bigr)$, so
$x_{1,e}=x_{2,e}=(\lambda-\mu)r_e$.
Suppose now that $b_e>0$, and set $q_e=\sqrt{2b_e/\lambda}$.
Direct substitution gives
\[
p_e(x_{1,e})
=\mu(\alpha b_e-q_ex_{1,e}).
\]
If $p_e(x_{1,e})\leq0$, then $x_{1,e}$ lies between the roots and hence
$x_{1,e}\leq x_{2,e}$. If $p_e(x_{1,e})>0$, then
\[
q_ex_{1,e}<\alpha b_e\leq b_e=\frac{\lambda q_e^2}{2},
\qquad\text{so}\qquad
x_{1,e}<\frac{\lambda q_e}{2}\leq\frac{S_e}{2}.
\]
Because $S_e/2$ is the vertex of the convex quadratic $p_e$, this places
$x_{1,e}$ to the left of the smaller root (or of the double root), and again
$x_{1,e}<x_{2,e}$. Finally, if
$\Delta_e=S_e^2-4\mu\alpha b_e\geq0$, then
\[
x_{2,e}=\frac{S_e+\sqrt{\Delta_e}}{2}
\leq S_e\leq T_e.
\]
Consequently, the capped domain gives
$0<x_{1,e}\leq x_{2,e}\leq T_e$ for every edge.

\begin{proposition}[Fixed-parameter quadratic majorization]
\label{prop:fixed-parameter-majorization}
Suppose that the stated cost assumptions and
\eqref{eq:general-admissible-mu} hold. Let $f^{(k)}\in\mathcal F$, let
$w_e^{(k)}=w_e(f_e^{(k)})$ be the exact weight given by
\eqref{eq:weight_ll} of the main paper, and let $f^{(k+1)}$ be the exact update (c.f. \eqref{eq:newf} of the main paper).
Then $W_k\succ0$, so
$\langle\cdot,\cdot\rangle_{W_k}$ is an inner product, and the model
$Q^{\lambda,\mu}(\,\cdot\mid f^{(k)})$ defined in
\eqref{eq:mm-majorizer-appendix} of the main paper is a global majorizer of the fixed-parameter
smoothed objective:
\begin{align}
J_{\mathrm{LL}}^{\lambda,\mu}(f)&\leq Q^{\lambda,\mu}(f\mid f^{(k)})
\quad\text{for every }f\in\mathbb R^m,
\label{eq:mm-global-bound-appendix}\\
Q^{\lambda,\mu}(f^{(k)}\mid f^{(k)})
&=J_{\mathrm{LL}}^{\lambda,\mu}(f^{(k)}).
\label{eq:mm-tangency-appendix}
\end{align}
The two functions also have the same gradient at $f^{(k)}$.
Consequently,
\begin{equation}
\label{eq:mm-descent-chain-appendix}
\begin{aligned}
J_{\mathrm{LL}}^{\lambda,\mu}(f^{(k+1)})
&\leq Q^{\lambda,\mu}(f^{(k+1)}\mid f^{(k)})\\
&\leq Q^{\lambda,\mu}(f^{(k)}\mid f^{(k)})
=J_{\mathrm{LL}}^{\lambda,\mu}(f^{(k)}).
\end{aligned}
\end{equation}
More precisely,
\begin{equation}
\label{eq:mm-sufficient-decrease-appendix}
\begin{aligned}
&J_{\mathrm{LL}}^{\lambda,\mu}(f^{(k)})
-J_{\mathrm{LL}}^{\lambda,\mu}(f^{(k+1)})\\
&\quad\geq
\frac{1}{2}
\left\langle
f^{(k+1)}-f^{(k)},f^{(k+1)}-f^{(k)}
\right\rangle_{W_k}.
\end{aligned}
\end{equation}
\end{proposition}

\begin{proof}
The ordering established above implies positivity of every exact analytical
weight. In the quadratic regime, $w_e(x)=1/(\lambda-\mu)>0$. In a nonempty
transition regime, $0<x<x_{2,e}\leq T_e$, so
$w_e(x)=(T_e-x)/(\mu x)>0$. In the active-flow regime,
$w_e(x)=r_e/x+\alpha b_e/x^2>0$ because $r_e+b_e>0$ and $\alpha>0$.
The weight at zero is defined by the quadratic branch. Hence $W_k\succ0$.

For each edge, define the scalar function
\begin{equation}
\label{eq:g-square-appendix}
g_e(t)=\mathrm{LL}_e^{\lambda,\mu}(\sqrt{t}),\qquad t\geq0.
\end{equation}
For $t>0$, differentiation of \eqref{eq:lasrylions} of the main paper gives
\begin{equation}
\label{eq:g-prime-appendix}
g_e'(t)=
\begin{cases}
\dfrac{1}{2(\lambda-\mu)},
&0<t<x_{1,e}^{2},\\[8pt]
\dfrac{T_e-\sqrt{t}}{2\mu\sqrt{t}},
&x_{1,e}^{2}<t<x_{2,e}^{2},\\[10pt]
\dfrac{r_e}{2\sqrt{t}}+\dfrac{\alpha b_e}{2t},
&t>x_{2,e}^{2}.
\end{cases}
\end{equation}
The values at the junctions are understood by continuous extension.
Inside the three respective regions,
\begin{equation}
\label{eq:g-second-appendix}
g_e''(t)=
\begin{cases}
0,
&0<t<x_{1,e}^{2},\\[3pt]
-\dfrac{T_e}{4\mu t^{3/2}},
&x_{1,e}^{2}<t<x_{2,e}^{2},\\[10pt]
-\dfrac{r_e}{4t^{3/2}}-\dfrac{\alpha b_e}{2t^2},
&t>x_{2,e}^{2}.
\end{cases}
\end{equation}
Since $r_e,b_e\geq0$, each expression in
\eqref{eq:g-second-appendix} is nonpositive. Moreover, the one-sided
derivatives agree at both transition points. Indeed,
\begin{equation}
\frac{T_e-x_{1,e}}{\mu x_{1,e}}
=\frac{1}{\lambda-\mu},
\end{equation}
while the defining equation
\begin{equation}
x_{2,e}^{2}-(T_e-\mu r_e)x_{2,e}+\mu\alpha b_e=0
\end{equation}
implies
\begin{equation}
\frac{T_e-x_{2,e}}{\mu x_{2,e}}
=\frac{r_e}{x_{2,e}}+\frac{\alpha b_e}{x_{2,e}^{2}}.
\end{equation}
If $x_{1,e}=x_{2,e}$, the middle interval is empty, and these identities
show that the derivatives of the quadratic and active-flow branches agree
directly at the common transition point. Value continuity follows from the
matched middle-branch value used to define the active-flow constant.
The right derivative at zero is
$g_e'(0)=1/[2(\lambda-\mu)]$. Hence $g_e$ is continuously
differentiable and concave on $[0,\infty)$.

The tangent inequality for a concave function now yields, for every
$f_e\in\mathbb R$,
\begin{align}
\mathrm{LL}_e^{\lambda,\mu}(f_e) =g_e(f_e^2)
&\leq
g_e\!\left((f_e^{(k)})^2\right)\nonumber\\
&\quad+
g_e'\!\left((f_e^{(k)})^2\right)
\left(f_e^2-(f_e^{(k)})^2\right).
\label{eq:edge-majorization-appendix}
\end{align}
By \eqref{eq:g-prime-appendix} and \eqref{eq:weight_ll} of the main paper,
\begin{equation}
2g_e'\!\left((f_e^{(k)})^2\right)=w_e^{(k)},
\end{equation}
where the first branch gives the same identity when $f_e^{(k)}=0$.
Summing \eqref{eq:edge-majorization-appendix} over the edges proves
\eqref{eq:mm-global-bound-appendix}. Substitution of $f=f^{(k)}$ gives
\eqref{eq:mm-tangency-appendix}, and
\[
\nabla_f Q^{\lambda,\mu}(f^{(k)}\mid f^{(k)})
=W_kf^{(k)}
=\nabla J_{\mathrm{LL}}^{\lambda,\mu}(f^{(k)}),
\]
which establishes first-order tangency.

It remains to prove the last claim. The terms in
\eqref{eq:mm-majorizer-appendix} of the main paper that are independent of $f$ do not
affect the minimizer, so \eqref{eq:newf} of the main paper is precisely the
IRLS weighted quadratic update. This gives the descent chain
\eqref{eq:mm-descent-chain-appendix}. Since $\mathcal F$ is convex,
optimality of $f^{(k+1)}$ for
$Q^{\lambda,\mu}(\,\cdot\mid f^{(k)})$ also gives
\[
\left\langle f^{(k+1)},
f^{(k)}-f^{(k+1)}\right\rangle_{W_k}\geq0.
\]
Expanding the quadratic difference therefore shows that
\[
\begin{aligned}
&Q^{\lambda,\mu}(f^{(k)}\mid f^{(k)})
-Q^{\lambda,\mu}(f^{(k+1)}\mid f^{(k)})\\
&\quad\geq
\frac{1}{2}
\left\langle
f^{(k+1)}-f^{(k)},f^{(k+1)}-f^{(k)}
\right\rangle_{W_k}.
\end{aligned}
\]
Combining this inequality with the majorization bound proves
\eqref{eq:mm-sufficient-decrease-appendix}.
\end{proof}

\paragraph{Idealized fixed-parameter convergence.}
If the exact update is iterated indefinitely at one fixed admissible pair
$(\lambda,\mu)$ satisfying \eqref{eq:general-admissible-mu},
then compactness of $\mathcal F$ and continuity of
$J_{\mathrm{LL}}^{\lambda,\mu}$ make the objective bounded below.
Summing \eqref{eq:mm-sufficient-decrease-appendix} gives
$\sum_k\|f^{(k+1)}-f^{(k)}\|^2_{W_k}<\infty$, so
$f^{(k+1)}-f^{(k)}\to0$ in the $W_k$-weighted norm, and standard MM
arguments imply that every accumulation point is stationary for
$J_{\mathrm{LL}}^{\lambda,\mu}$ over $\mathcal F$
\cite{Razaviyayn2013unified}. This idealized statement does not assert
convergence of a finite continuation stage or of the full parameter-varying
algorithm.

\paragraph{Numerical qualification.}
\Cref{prop:fixed-parameter-majorization} concerns a fixed admissible pair,
the exact analytical weights in \eqref{eq:weight_ll} of the main paper, and an exact solution
of the weighted quadratic subproblem. If a flow magnitude is floored or a
weight is clipped for numerical protection, the resulting quadratic need
not retain either global majorization or first-order tangency. With exact
weights, an inexact subproblem solve retains the descent conclusion whenever
$Q^{\lambda,\mu}(f^{(k+1)}\mid f^{(k)})
\leq Q^{\lambda,\mu}(f^{(k)}\mid f^{(k)})$. 
Because the implementation clips weights and solves the QP numerically, it
does not infer descent from the proposition. Instead it evaluates
$J_{\mathrm{LL}}^{\lambda,\mu}$ directly at the fixed smoothing pair after a
trial step and backtracks or rejects any step that fails the strict acceptance rule specified in \Cref{app:algorithmic_details}.

\paragraph{Scope with respect to continuation.}
The proposition is step-wise: it compares consecutive iterates using the
same fixed pair $(\lambda,\mu)$. It does not imply a comparison after a
continuation update. With the parameter dependence displayed explicitly,
the jump
\[
J_{\mathrm{LL}}^{\lambda_{k+1},\mu_{k+1}}(f^{(k+1)})
-J_{\mathrm{LL}}^{\lambda_k,\mu_k}(f^{(k+1)})
\]
has no fixed sign in general. Non-positivity of this quantity is desirable, but cannot be guaranteed as downstream effect of the particular shape of the Lasry--Lions functional $\mathrm{LL}^{\lambda,\mu}$ and the continuation scheme for $\lambda$ and $\mu$.

\subsection{Weighted-QP Solver}
\label{app:dualssn}

This section documents the weighted-QP solver of \Cref{sec:dual_ssn} in full: the dual semismooth Newton (SSN) scheme implemented, the empirical scaling of its inner iteration counts across problem sizes, and the primal-dual interior-point method (IPM) used as a fallback when SSN fails to converge.

\paragraph{Dual semismooth Newton.}

Each weighted-QP solve is initialized with the dual potential and active-set masks from the previous accepted outer iterate. On the first call, no such state is available, so the input dual potential is set to $y=0$ and the masks are left empty. The solver does not form a Newton system at this initial value: since $D(0)=0$, a literal Newton step from $y=0$ would make \eqref{eq:ssnstep} of the main paper degenerate. Instead, it first constructs a dual predictor from the current primal flow $f$, updates the corresponding active-set masks, and then applies the semismooth Newton correction.

\paragraph{Predictor.}
The free and bound-active arcs are first estimated directly from the current primal flow ($f$) by thresholding it against the lower and upper bounds, rather than from the dual potential ($y$). If the resulting free subgraph is component-wise balanced with respect to the residual demand induced by the bound-active arcs, the solver performs one Laplacian solve on that subgraph. The resulting primal--dual candidate is accepted without any semismooth Newton iterations if it is primal feasible and the recovered dual variables satisfy the sign conditions associated with the estimated free and bound-active sets.

\paragraph{Regularized correction.}
If the predictor is rejected or infeasible, the solver returns to the warm-started dual potential $y$ and applies semismooth Newton corrections. At each step, arcs are classified as free, bound-active, or near a kink, with the latter lying within a numerical tolerance of either bound. On connected components for which the free and bound-active partition is component-wise balanced, the solver uses the exact generalized derivative $D(y)$. Otherwise, it regularizes the derivative by adding a penalty $\eta/w_e$ on near-kink and bound-active arcs. The regularization parameter is initialized at $\eta=10^{-2}$ and increased adaptively if the linear system cannot be solved accurately, the computed direction does not reduce the dual residual, or the free-subgraph solve remains poorly converged. The regularized Newton matrix is again a graph Laplacian and therefore remains singular up to an additive constant on each connected component. For a component-wise compatible right-hand side, the Newton equation is understood modulo these gauges (equivalently, through a Laplacian pseudoinverse); the implementation mean-centers the right-hand side and returned potential, which fixes the gauge when the regularized conductance graph is connected. The regularization still provides an inexact direction when the unregularized free-subgraph system is degenerate.

\paragraph{Line search.}
Let $\phi(\beta)$ denote the concave dual objective restricted to the search
path $y+\beta\Delta y$. Its directional derivative is
\[
    \phi'(\beta)=\Delta y^\top R(y+\beta\Delta y).
\]
This scalar function is continuous, nonincreasing, and piecewise linear in
$\beta$. The solver accepts the full step if $\phi'(1)\geq 0$; otherwise, it
locates a zero of $\phi'$ in $(0,1)$ before updating the dual potential,
thereby avoiding overshooting an intervening kink.

\paragraph{Inner iteration counts.}

Each outer IRLS iteration issues one weighted-QP solve, and each dual semismooth Newton step inside that solve requires one PCG solve of a free-subgraph Laplacian system, using a freshly computed approximate-Cholesky factorization
as preconditioner, at cost $\widetilde O(m\log(1/\epsilon))$.
Writing $K$ for the number of outer iterations and $\bar s$ for the average
number of Newton steps per subproblem, the dominant cost of a single IRLS run is
$\widetilde O(K\,\bar s\,m\log(1/\epsilon))$. Since $K$ is capped at $200$
by construction, the practical question is how $\bar s$ behaves as $m$ grows.

\Cref{tab:newton_scaling} reports these counters for one Layer-1 run on an
instance from each of the nine \texttt{scale} families,
which span a $200$-fold range in $m$ and account for $270$ of the $410$ instances. All nine runs terminate normally, so the counts are complete. Over this range $\bar s$ grows slowly, from $1.00$ at $m=100$ to $2.91$ at $m=2\times10^{4}$, and never exceeds $3.59$; a least-squares fit gives $\bar s\approx-1.03+0.41\log m$, so a $200$-fold increase in $m$ costs less than a threefold increase in Newton steps per subproblem. Since $\bar s$ itself grows only logarithmically in $m$, the per-run cost over these instances is therefore close to $\widetilde O(K\,m)$, with the $\log(1/\epsilon)$ and $\log m$ factors both absorbed into the $\widetilde O$. The number of outer iterations is likewise small, between $12$ and $26$, so no run approaches the iteration cap. On every instance the number of Laplacian solves exceeds the number of Newton steps by exactly $K$: the warm-start predictor
contributes one additional solve per outer iteration, and no subproblem requires a restart beyond it.

These figures are for a single instance per family and for Layer 1 alone; Layers
2 and 3 invoke the same solver repeatedly, so their totals scale with the number
of restarts and restricted supports rather than with $\bar s$.

\begin{table}[t]
\centering
\caption{Inner iteration counts for one Layer-1 IRLS run per \texttt{scale}
family. $K$ is the number of outer iterations, which equals the number of
weighted-QP solves; $\bar s$ is the average number of Newton steps per solve;
Lapl.\ counts Laplacian solver calls.}
\label{tab:newton_scaling}
\small
\setlength{\tabcolsep}{5pt}
\begin{tabular}{rrrrrr}
\toprule
$n$ & $m$ & $K$ & Newton & $\bar s$ & Lapl. \\
\midrule
$20$      & $100$      & 12 & 12 & 1.00 & 24 \\
$60$      & $400$      & 12 & 17 & 1.42 & 29 \\
$120$     & $1{,}500$  & 20 & 23 & 1.15 & 43 \\
$150$     & $2{,}500$  & 15 & 31 & 2.07 & 46 \\
$200$     & $4{,}000$  & 17 & 61 & 3.59 & 78 \\
$250$     & $7{,}500$  & 21 & 53 & 2.52 & 74 \\
$300$     & $9{,}000$  & 26 & 57 & 2.19 & 83 \\
$500$     & $10{,}000$ & 21 & 65 & 3.10 & 86 \\
$1{,}000$ & $20{,}000$ & 22 & 64 & 2.91 & 86 \\
\bottomrule
\end{tabular}
\end{table}

\paragraph{Interior-point method formulation.}
When the dual SSN solver described above fails to converge to a flow-conservation residual below $10^{-6}$, the same weighted quadratic subproblem is instead solved by a primal-dual interior-point method (IPM), which shares the weighted graph-Laplacian structure of the SSN Newton systems. The weighted quadratic subproblem can be written as
\begin{equation}\label{eq:qp_subproblem_2_appendix}
    \begin{aligned}
\min_f \quad & \frac{1}{2} f^\top W f \\
\text{s.t.}\quad & Cf=d, \\ & 0 \leq f \leq c.
\end{aligned}
\end{equation}

To solve \eqref{eq:qp_subproblem_2_appendix}, we use a primal-dual interior-point method \cite{nocedal2006numerical}. Introducing dual variables \(y\) for the equality constraints and nonnegative dual variables \(s\) and \(q\) for the lower and upper bounds, respectively, the Lagrangian is
\begin{equation}\label{eq:lagrangian_appendix}
\mathcal{L}(f,y,s,q) = \frac{1}{2}f^\top W f -y^\top(Cf-d) - s^\top f - q^\top(c-f).
\end{equation}

The KKT residuals are
\begin{equation}\label{eq:kkt1_appendix}
r_b = Cf-d,
\end{equation}
and
\begin{equation}\label{eq:kkt2_appendix}
r_c = -C^\top y + Wf - s + q.
\end{equation}

The complementarity conditions are
\begin{equation}\label{eq:dual1_appendix}
f_e s_e = 0,
\end{equation}
and
\begin{equation}\label{eq:dual2_appendix}
(c_e-f_e)q_e = 0.
\end{equation}

In the interior-point method, \eqref{eq:dual1_appendix} and \eqref{eq:dual2_appendix} are relaxed as
\begin{equation}
f_e s_e \approx \beta,
\end{equation}
and
\begin{equation}
(c_e-f_e)q_e \approx \beta,
\end{equation}
where \(\beta>0\) is the barrier target. Then, we define the complementarity residuals as
\begin{align}
r_{c1} &= f\odot s - \beta\mathbf{1},\\
r_{c2} &= (c-f)\odot q - \beta\mathbf{1}.
\end{align}

Linearizing the perturbed KKT system and eliminating the slack directions yields the reduced system
\begin{equation}
\left(C\Theta C^\top\right)\Delta y = -C(\Theta \zeta)-r_b,
\end{equation}
where 
\begin{equation}
\zeta = -r_c - F^{-1}r_{c1} + (U-F)^{-1}r_{c2}
\end{equation} and 
\begin{equation}
\Theta = H^{-1}.
\end{equation}

The Newton system involves the diagonal matrix
\begin{equation}
H = W + F^{-1}S + (U-F)^{-1}Q,
\end{equation}
where
\begin{equation}
F=\operatorname{diag}(f),
S=\operatorname{diag}(s),
Q=\operatorname{diag}(q),
U=\operatorname{diag}(c).
\end{equation}

$C\Theta C^\top$ is a weighted graph Laplacian, so the method can use Laplacian solvers to accelerate
the subproblem solve.

Once \(\Delta y\) is computed, the remaining directions are recovered from
\begin{align}
\Delta f &= \Theta\left(\zeta + C^\top \Delta y\right),\\
\Delta s &= F^{-1}\left(-r_{c1} - S \Delta f\right),\\
\Delta q &= (U-F)^{-1}\left(-r_{c2} + Q \Delta f\right).
\end{align}
A step length is then selected to preserve strict positivity of the primal-dual variables, and the iterates are updated accordingly.

The overall method has a nested structure. The outer loop corresponds to the IRLS reweighting step, while the inner loop solves the resulting convex quadratic subproblem.

At outer iteration \(k\), given the current flow \(f^{(k)}\), we compute the edge weights
\begin{equation}
w_e^{(k)} = w_{\lambda_k,\mu_k}(f_e^{(k)}), \qquad \forall e\in E,
\end{equation}
using \eqref{eq:weight_ll} of the main paper. These weights define the quadratic model \eqref{eq:qp_subproblem_2_appendix}, 
which is then solved by the interior-point method to obtain \(f^{(k+1)}\). The surrogate parameters are subsequently updated, producing a new smoothed model for the next outer iteration.

Hence, the outer loop is responsible for approximating the original non-convex objective, while the inner loop computes the minimizer of the current quadratic model. For exact analytical weights in the capped domain, this model is the majorizer of \Cref{prop:fixed-parameter-majorization}; the clipped numerical model is safeguarded by direct surrogate-value acceptance.

\subsection{Algorithmic Details}
\label{app:algorithmic_details}

This section completes the specification of the IRLS core method of \Cref{sec:algorithm}: the support-refit mechanism, termination conditions, numerical clamping, the step-acceptance rule, stagnation tests, and candidate feasibility evaluation.

\paragraph{Support refit.}
The refit is attempted when $\rho_{\text{gap}}\ge10$ and $0<K_{\text{gap}}\le\max\{4n,1024\}$. Its threshold $\sqrt{f_\uparrow f_\downarrow}$ is the midpoint of the gap on a logarithmic scale. An attempt is skipped when the support repeats the previous one, and a cost model is discarded when the restricted network cannot route the demand.  We solve it under two cost models, $r_e$ and $r_e+b_e/\max\{f_e,1\}$, and keep the smaller $J_{\mathrm{true}}$: the former minimizes variable cost on $S$, while the latter rewards closing weakly loaded arcs and remains informative when $r\equiv0$. Both models use the same integer-cost network-simplex routine as the stage-boundary candidate. The result enters the best feasible solution only if it improves the incumbent by more than $\max\{10^{-9},10^{-6}|J^{\star}_{\mathrm{true}}|\}$, so the refit cannot degrade the reported objective. The final refit uses the visible support $\{e:|f_e|>\tau\}$ of the last iterate.

Each stage boundary additionally produces a full-graph feasible candidate, obtained by solving a min-cost flow with unit costs $r_e+b_e/\max\{f_e,1\}$ on the active arcs and $r_e+b_e/\max\{c_e,1\}$ elsewhere. This candidate is kept only if it is feasible and improves $J_{\mathrm{true}}$.

\paragraph{Termination.}
The run exits once $\rho_{\text{gap}}\ge100$, $0<K_{\text{gap}}\le\max\{4n,1024\}$, $K_{\text{gap}}$ has taken the same value on four consecutive iterations, and the best feasible solution has not improved for eight iterations. The separation required here is two decades rather than the factor of five needed to trust $\lambdagap$, because exit is irreversible while an inaccurate target only perturbs the rate of sharpening. The $1024$ floor keeps this trigger permissive on small networks, where $4n$ alone would be too restrictive; for larger graphs ($n>256$) the bound scales as $4n$, a small multiple of the $n-1$ arcs of a basic feasible flow. The run also stops when the surrogate fails to decrease for $15$ iterations with relative flow change below $10^{-4}$, when $\lambda$ reaches $\lambda_{\min}$, or at the limit of $200$ reweighting iterations or the time budget.
 
\paragraph{Clamping.}
The IRLS weights of \eqref{eq:weight_ll} of the main paper are clipped to $[10^{-12},10^{12}]$ before each solve, and the variable costs entering the surrogate, the weights, and $\mu_{\max,e}$ are floored at $r_e\leftarrow\max\{r_e,10^{-12}\}$. The floor matters on the pure fixed-charge families, for which $r\equiv0$ and several surrogate quantities would otherwise be degenerate. Objectives are evaluated with the unmodified costs, so the reported $J_{\mathrm{true}}$ is unaffected. As noted after \Cref{prop:fixed-parameter-majorization}, a clipped weight may forfeit global majorization, which is why acceptance of a step is tested directly rather than assumed.
 
\paragraph{Step acceptance.}
A trial iterate $\tilde f$ is evaluated using
$J_{\mathrm{LL}}^{\lambda_k,\mu_k}(\cdot)$ at the current parameter pair and accepted only if
\[
\begin{aligned}
    J_{\mathrm{LL}}^{\lambda_k,\mu_k}(\tilde f)&\le J_{\mathrm{LL}}^{\lambda_k,\mu_k}(f^{(k)})
\end{aligned}
\]
(the numerical tolerance mentioned in the main paper is set to $0$ in the implementation). Otherwise the step is backtracked along $f^{(k)}+\beta(\tilde f-f^{(k)})$ for $\beta\in\{\tfrac12,\tfrac14,\tfrac18,\tfrac1{16},\tfrac1{32}\}$, and the first $\beta$ satisfying the same test is taken. If none does, the step is rejected and $f^{(k+1)}=f^{(k)}$. Whenever $\beta<1$, the dual warm start is discarded and the semismooth Newton state is reinitialized to $y=0$ and $s=q=\mathbf 1$, because the retained active set no longer matches the accepted iterate.
 
\paragraph{Stagnation tests.}
Both counters use a relative threshold of the same form. The surrogate counter advances unless $J_{\mathrm{LL}}^{\lambda,\mu}$ decreases by more than $\max\{10^{-9},10^{-6}|J_{\mathrm{LL}}^{\lambda,\mu}(f^{(k)})|\}$. The true-objective counter advances unless a feasible candidate improves the best value by more than $\max\{10^{-9},10^{-6}|J^{\star}_{\mathrm{true}}|\}$.
 
\paragraph{Candidate evaluation.}
All candidates are scored on the visible flow obtained by setting to zero every arc with $f_e\le\tau=10^{-6}$. A candidate is feasible when the visible flow satisfies $\|Cf-d\|_\infty\le10^{-6}$, together with capacity and nonnegativity violations of at most $10^{-6}$, and only feasible candidates are retained. Applying the threshold before the feasibility test is what makes an interior iterate ineligible, and is the reason the stage-boundary candidates are needed. Those candidates are passed to an integer-cost network-simplex routine, so arc costs are scaled by $10^{6}$ and rounded. A candidate is skipped when the node balances are not integral to within $10^{-6}$ or do not sum to zero, which does not occur on the families reported here.

\subsection{Smoothing Continuation: Implementation Details}
\label{app:smoothing_details}
 
This section completes the specification of \Cref{sec:lambda_update}. \Cref{tab:smoothing_hparams} collects the constants of the schedule, which were held fixed across all reported runs. These constants were fixed once via manual tuning on a small subset of instances prior to the full $410$-instance run, rather than via a systematic search over a stated range of values.
 
\paragraph{Initialization.}
Let $f_{\mathrm{med}}$ be the median magnitude among initial flows exceeding one unit, and let $b_{\mathrm{med}}$ be the median positive fixed cost. The choice $\lambda=f_{\mathrm{med}}^{2}/(2b_{\mathrm{med}})$ places $T_e$ at $f_{\mathrm{med}}$ when $r_e$ is negligible. Starts that already carry combinatorial support information, namely those derived from MCF, DSSP, and ADSSP solutions and from restart or cached supports, begin one decade below this level in flow units, and therefore two decades below in $\lambda$ because $T_e\propto\sqrt{\lambda}$:
\[
\lambda_0=10^{-2}\max\left\{\frac{f_{\mathrm{med}}^{2}}{2b_{\mathrm{med}}},\;1\right\}.
\]

The active flows of the start then lie in the logarithmic tail, so that sharpening acts on the incumbent support rather than dissolving it. If fewer than two arcs carry flow, this scale is undefined and we use $\lambda_0=10^{-2}\max\{3J_0,1\}$ instead, where $J_0$ is the realized cost of the start. Diffuse quadratic-programming starts have no support worth preserving, and are deliberately over-smoothed with $\lambda_0=10^{2}\max\{\max_e|f^{(0)}_e|,1\}$. This places every arc in the quadratic head, so the first subproblem is an almost uniformly weighted projection onto $\mathcal F$. The initial ratio is $\mu_0=\min\{\theta_0\lambda_0,\min_e\mu_{\max,e}(\lambda_0)\}$ with $\theta_0=0.3$ and $\theta_0=0.7$ in the two cases. The cap is inactive for $\theta_0=0.3$ but can reduce the diffuse-start ratio below $0.7$; both choices satisfy the capped domain of \Cref{prop:fixed-parameter-majorization}. The special $\theta_0$ applies to the initial weighted-QP solve, after which every continuation update uses $\mu=\lambda/2$. The log-tail coefficient is anchored to the start through $\alpha=\lambda/(\lambda+s\lambda_0)$ with $s=10^{-2}$, so that $\alpha\approx0.99$ initially and $\alpha=1/2$ after two decades of continuation, past which the tail relaxes toward the exact Lasry--Lions envelope recovered as $\alpha\to0$.
 
\paragraph{Gap statistic.}
The sorted flow magnitudes are floored at $\lambda_{\min}$ before the ratios are formed, and we write $f_\uparrow=|f|_{(K_{\text{gap}})}$ and $f_\downarrow=|f|_{(K_{\text{gap}}+1)}$ for the magnitudes bracketing the gap. The representative fixed cost $b_{\mathrm{rep}}$ is the flow-weighted average of $b_e$ over the arcs within one octave of the gap, that is $|f_e|\in[f_\downarrow/2,\,2f_\uparrow]$ with $b_e>10^{-3}$, and the median positive $b_e$ when that set is empty. The gap is trusted, and $\lambdagap$ used, only when $\rho_{\text{gap}}\ge5$, $K_{\text{gap}}\le m/2$, and $K_{\text{gap}}$ has taken the same value on three consecutive iterations. The last two conditions reject a ratio that lies in the interior of a broad flow distribution.
 
\paragraph{Stage rule.}
A stage ends when the relative flow change falls below $10^{-3}$, when the line search rejects the step, or after six weighted-QP solves. At a stage boundary,
\[
\lambda^{+}=\max\left\{\min\{\lambda/4,\;\lambdagap\},\;\lambda/100\right\},
\]
while between boundaries a trusted target may cap the level by at most a factor of five, that is $\lambda^{+}=\max\{\min\{\lambda,\lambdagap\},\lambda/5\}$. After flooring at $\lambda_{\min}$, we set $\mu^{+}=\min\{\lambda^{+}/2,\min_e\mu_{\max,e}(\lambda^{+})\}$. The second bound is never active, since $\mu_{\max,e}(\lambda)\ge\lambda/2$ holds for every edge whenever $\alpha\le1$.

\begin{table}[t]
\centering
\caption{Constants of the smoothing continuation schedule.}
\label{tab:smoothing_hparams}
\small
\setlength{\tabcolsep}{5pt}
\begin{tabular}{@{}p{0.48\columnwidth}p{0.46\columnwidth}@{}}
\toprule
Quantity & Value \\
\midrule
Post-initialization $\theta=\mu/\lambda$ & $1/2$ \\
Initial-QP requested $\theta_0$ & $0.3$ / $0.7$ \\
$\lambda_0$, sparse start & $10^{-2}\max\{f_{\mathrm{med}}^{2}/(2b_{\mathrm{med}}),1\}$ \\
$\lambda_0$, dense start & $10^{2}\max\{\max_e|f^{(0)}_e|,1\}$ \\
$s$ (in $\alpha$) & $10^{-2}$ \\
$\lambda_{\min}$ & $10^{-9}$ \\
Stage decay & $\lambda/4$ \\
Stage clamp & $\lambda/100$ \\
Intra-stage clamp & $\lambda/5$ \\
Stage QP cap & $6$ solves \\
Stage-close flow change & $10^{-3}$ \\
Trust: $\rho_{\text{gap}}$, $K_{\text{gap}}$, stability & $5$, $m/2$, $3$ iters \\
Refit trigger: $\rho_{\text{gap}}$, $K_{\text{gap}}$ & $10$, $\max\{4n,1024\}$ \\
Refit support threshold & $\sqrt{f_\uparrow f_\downarrow}$ \\
Refit cost models & $r_e$ and $r_e+b_e/\max\{f_e,1\}$ \\
Exit: $\rho_{\text{gap}}$, $K_{\text{gap}}$, stability & $100$, $\max\{4n,1024\}$, $4$ iters \\
Exit patience & $8$ iters \\
Surrogate stagnation & $15$ iters, $10^{-4}$ \\
$b_{\mathrm{rep}}$ window & $[f_\downarrow/2,\,2f_\uparrow]$, $b_e>10^{-3}$ \\
Weight clip & $[10^{-12},10^{12}]$ \\
Variable-cost floor & $10^{-12}$ \\
Backtracking factors & $\tfrac12,\tfrac14,\tfrac18,\tfrac1{16},\tfrac1{32}$ \\
Activation threshold $\tau$ & $10^{-6}$ \\
Feasibility tolerance & $10^{-6}$ \\
Refit cost scaling & $10^{6}$ \\
Reweighting iteration limit & $200$ \\
\bottomrule
\end{tabular}
\end{table}

\subsection{Algorithmic Layers}
\label{app:algo_layers}

This section completes the specification of the three algorithmic layers of \Cref{sec:algo_layers}: initialization strategies for multi-start IRLS, the objective-driven perturbation-restart mechanism, and the implementation details of the anchor-union restricted search.

\subsubsection{Initialization Strategies}
\label{app:initializations}

The multi-start IRLS layer uses several feasible initial flows. Quadratic warm starts solve weighted least-squares flow problems over \(\mathcal F\).
The linear-cost-weighted quadratic flow initialization solves
\[
\min_{f\in\mathcal F}
\sum_{e\in E} r_e f_e^2.
\]
The quadratic initialization weighted by combined linear and capacity-amortized fixed cost solves
\[
\min_{f\in\mathcal F}
\sum_{e\in E}
\left(
r_e+\frac{b_e}{\max\{c_e,1\}}
\right)f_e^2.
\]

The min-cost-flow warm starts solve linearized flow problems with combined per-unit costs. One start uses total-flow-amortized fixed charges, while the capacity-aware start uses costs
\[
r_e+\frac{b_e}{\max\{c_e,1\}},
\]
which favors high-capacity arcs with low amortized fixed charge. We also include a DSSP-based start obtained from the slope-scaling iterate. All initializations are checked for feasibility, and duplicate supports are removed before launching IRLS.

\subsubsection{Objective-Driven Perturbation Restarts}
\label{app:perturbation_restarts}

Although the multi-start strategy improves robustness, each IRLS run still evolves locally from its initial support. Once the smoothing parameters become small, the surrogate strongly penalizes opening new arcs, and the method may continue improving the flow values without substantially changing the active arc set. To encourage larger support changes, we introduce an objective-driven perturbation restart layer.

Let \(f\) denote an incumbent feasible flow returned by IRLS. We define the active and inactive arc sets as
\[
A(f)=\{e\in E:f_e>\tau\},
\qquad
I(f)=\{e\in E:f_e\le \tau\},
\]
where \(\tau>0\) is the activation threshold used for scoring. Active arcs are ranked by their realized average cost
\[
\sigma_e^{\mathrm{act}}
=
r_e+\frac{b_e}{\max\{f_e,1\}},
\qquad e\in A(f),
\]
so arcs carrying little flow but incurring large fixed cost are treated as expensive. Inactive arcs are ranked by their capacity-amortized opening cost
\[
\sigma_e^{\mathrm{inact}}
=
r_e+\frac{b_e}{\max\{c_e,1\}},
\qquad e\in I(f),
\]
so arcs with low variable cost and low amortized fixed cost are treated as promising alternatives.

Using these scores, we generate perturbed target vectors that modify the incumbent support in several complementary ways. Weak-arc deletions reduce flow on active arcs carrying very small amounts. Expensive-arc deletions reduce or remove active arcs with large \(\sigma_e^{\mathrm{act}}\). Cheap-arc openings assign trial flow to inactive arcs with small \(\sigma_e^{\mathrm{inact}}\). Rerouting perturbations simultaneously suppress expensive active arcs and open cheap inactive arcs. Support swaps replace weak active arcs with a small set of cheap inactive alternatives.

A perturbed target \(\tilde f\) is not required to satisfy flow conservation or capacity constraints. For example, opening a cheap inactive arc does not guarantee that a complete feasible route exists through that arc. Therefore, the target is not accepted directly. Instead, it is projected onto the feasible capacitated-flow polytope
\[
\mathcal F=\{g\in\mathbb{R}^{|E|}:Cg=d,\;0\le g\le c\}.
\]
The projected restart point is obtained from
\[
\Pi_{\mathcal F}(\tilde f)
=
\arg\min_{g\in\mathcal F}
\|g-\tilde f\|_2^2.
\]
The restart initialization is then \(f_{\mathrm{restart}}^{(0)}=\Pi_{\mathcal F}(\tilde f)\). This projection allows the perturbation step to be aggressive at the support level while restoring feasibility through a controlled continuous optimization problem.

If the projection fails, produces only a negligible change from the incumbent, or duplicates a support already generated by another perturbation, the candidate is discarded. The perturbed targets are first ranked by their realized objective value and support size. After projection and duplicate filtering, IRLS is rerun from the surviving restart candidates, and the best resulting flow is selected lexicographically by objective value, feasibility residual, and runtime.

\subsubsection{Layer-3 Implementation Details}
\label{app:layer3_details}

This section details the anchor-union restricted search of \Cref{sec:algo_layers}: how anchor supports are generated, how candidate unions are formed and filtered, the restricted refit portfolio, restricted IRLS runs, runtime control, within-union refinement, and final candidate selection.

\paragraph{Anchor generation.}
The anchor set \(\mathcal X_{\mathrm{anc}}\) is generated using ADSSP runs with four initializations:
\texttt{max\_flow}, \texttt{unitflow}, \texttt{trans}, \texttt{multi\_unit}. Each anchor run is limited to \(120\) iterations. Only feasible returned flows are retained as anchors.

\paragraph{Union supports.}
For each restart-layer candidate \(f\), the implementation constructs three types of restricted supports: $S(f)$,
$S(f)\cup S(a)$ for each $a\in\mathcal X_{\mathrm{anc}}$, and $S(f)\cup\bigcup_{a\in\mathcal X_{\mathrm{anc}}}S(a)$.

Duplicate and empty supports are removed. A support is also discarded if the
restricted directed graph fails a reachability screening test from supply nodes
to demand nodes.

\paragraph{Restricted refit portfolio.}
For each retained support \(E'\), feasible warm starts are obtained from two sources: the corresponding Layer 2 candidate flow restricted to \(E'\), and restricted linear flow problems
\begin{equation}
\label{eq:app_l3_refit}
\begin{aligned}
\min_{g}\quad
&
\sum_{e\in E'}
\left(
r_e+\gamma\frac{b_e}{\max\{c_e,1\}}
\right)g_e,\\
\mathrm{s.t.}\quad
&
C_{E'}g=d,\\
&
0\leq g_e\leq c_e,\qquad e\in E',
\end{aligned}
\end{equation}
with
\[
\gamma\in\{0,\;0.05,\;0.25,\;1,\;4\}.
\]
The case \(\gamma=0\) gives a pure variable-cost refit, while positive values
penalize arcs with large fixed cost relative to capacity. These objectives are used only to generate feasible starts; all candidates, including the restricted Layer 2 flows, are scored by \(J_{\mathrm{true}}\). If no feasible refit is found, the support is discarded. Among all feasible starts, duplicate numerical supports are removed and at most the three best starts are retained, ordered by \(J_{\mathrm{true}}\) and then by support size.

\paragraph{Restricted IRLS runs.}
Restricted IRLS uses the same smoothing schedule, weighted-QP solver,
support-refit rule, feasibility tolerance, and stopping criteria as the main
IRLS method.  The returned restricted flow is
lifted to the original graph by assigning zero flow to arcs outside \(E'\).

\paragraph{Support selection and runtime control.}
Candidate supports are ranked by the true objective value of their best feasible
refit, with support size used as a secondary criterion. To control runtime, only
a bounded number of supports are processed:
\[
N_{\mathrm{union}}
=
\begin{cases}
2, & n\ge4000\ \text{or}\ m\ge200000,\\
3, & n\ge1500\ \text{or}\ m\ge75000,\\
6, & \text{otherwise}.
\end{cases}
\]
The remaining Layer-3 time budget is divided among the selected supports.

\paragraph{Within-union refinement.}
After restricted IRLS, the algorithm performs a monotone support refinement
inside \(E'\). Active arcs are ranked for deletion using
\begin{equation}
\label{eq:app_l3_delete_score}
q_e^{\mathrm{del}}
=
r_e
+
\frac{b_e}{\max\{f_e,10^{-6}\}}
+
\frac{1}{2}\frac{b_e}{\max\{c_e,1\}}.
\end{equation}
Deletion batches of approximately \(12\%\) and \(6\%\) of the active support
are tested first, followed by individual trials among the \(25\) highest-ranked
active arcs. At most five successful deletion rounds are performed.

Unused arcs in \(E'\) are ranked for reopening using
\begin{equation}
\label{eq:app_l3_add_score}
q_e^{\mathrm{add}}
=
r_e+\frac{b_e}{\max\{c_e,1\}}.
\end{equation}
The algorithm tests bundles of up to \(8\), \(4\), and \(2\) arcs, followed by
individual trials among the \(20\) highest-ranked unused arcs. At most three
successful reopening rounds are performed.

Each deletion or reopening trial is followed by the restricted refit portfolio
in \eqref{eq:app_l3_refit}. A modification is accepted only if the refit is
feasible and strictly decreases \(J_{\mathrm{true}}\). Thus, this refinement
phase is monotone with respect to the reported fixed-charge objective.

\paragraph{Final selection.}
Layer 3 returns the feasible candidate with the smallest \(J_{\mathrm{true}}\)
among the restart-layer candidates, restricted refits, restricted IRLS outputs,
and within-union refinement outputs.

\subsection{Experiments}\label{app:experiment_setup}
This section gives the hardware and software environment, random seeds, and instance-generation parameters underlying the results in \Cref{sec:experiments}.

\paragraph{Hardware and Software.}
All experiments were run on a MacBook Pro (Apple M4 Pro, 48\,GB RAM) under macOS Sequoia \texttt{15.7.5}. Methods were implemented in Python \texttt{3.14.3}. Min-cost-flow and MILP subproblems use OR-Tools \texttt{9.15.6755}. The MILP baseline is formulated through the OR-Tools \texttt{pywraplp} interface and solved with its \texttt{SAT} backend. For consistency across methods, reported objective values are obtained by re-evaluating the returned flow under the original fixed-charge objective. Reported times are wall-clock, single-run, on an otherwise idle machine.

\paragraph{Random seeds.}
All reported heuristic runs use a master seed of $0$ (default). Per-initialization seeds within a run are derived deterministically from the master seed (e.g.\ offsets of the form $\mathrm{seed}+\mathrm{idx}$ for multi-start IRLS and ADSSP initializations), so a fixed master seed reproduces the full 410-instance run exactly. 

\paragraph{Instance generation.}
The \texttt{synthetic} family is generated by a parametric single-commodity generator with per-family specifications of node count, arc density, supply/demand node counts, capacity factor, and fixed-cost ratio, drawn with \texttt{numpy}'s \texttt{default\_rng(seed)} for reproducibility. The ten synthetic subfamilies vary these parameters (e.g.\ \texttt{high\_fixed\_cost}, \texttt{low\_fixed\_cost}, \texttt{tight\_capacity}) to probe different regimes of the fixed-charge-to-variable-cost ratio and capacity tightness, with $10$ instances drawn per subfamily. The \texttt{scale\_*} families follow the generation procedure of \cite{yang2024sequential}. \texttt{FCNetLib\_fc} and \texttt{Benchmark} are fixed, externally sourced instance sets and are not randomly generated.

\subsection{Additional Experimental Results}
\label{app:results}
This section supports \Cref{sec:experiments} with the full per-family breakdown, the split by MILP termination status, and the anchor-only ablation referenced there.
\begin{table*}[t]
\centering
{
\caption{Mean gap (\%) to the MILP baseline and mean time (s) by
family.  The best heuristic mean gap per family is bold.  A negative entry
indicates L3 improves on a time-limited MILP incumbent.}
\label{tab:familygaps}
\scriptsize
\setlength{\tabcolsep}{1.4pt}
\begin{tabular}{lrrrrrrrrrrrrrr}
\toprule
 & & \multicolumn{6}{c}{Mean gap (\%)} & \multicolumn{7}{c}{Time (s)} \\
\cmidrule(lr){3-8}\cmidrule(lr){9-15}
Family & \# instances & L3 & ADSSP & DSSP & L2 & L1 & MCF & MILP & L3 & L2 & L1 & ADSSP & DSSP & MCF \\
\midrule
FCNetLib\_fc & 20 & \textbf{6.972} & 12.881 & 12.947 & 12.227 & 17.312 & 39.067 & 2.33 & 25.53 & 19.11 & 5.10 & 0.03 & 0.04 & 0.01 \\
Benchmark & 20 & \textbf{7.642} & 9.462 & 10.270 & 11.934 & 13.363 & 21.392 & 155.96 & 10.15 & 5.54 & 3.74 & 0.15 & 0.63 & 0.03 \\
synthetic & 100 & \textbf{1.706} & 4.766 & 6.089 & 2.688 & 4.359 & 10.979 & 75.85 & 7.30 & 3.17 & 0.89 & 0.15 & 0.78 & 0.04 \\
sc.20\_100 & 30 & \textbf{0.014} & 0.378 & 0.550 & 0.132 & 0.406 & 0.627 & 0.05 & 0.30 & 0.12 & 0.04 & 0.01 & 0.01 & 0.00 \\
sc.60\_400 & 30 & \textbf{0.348} & 1.486 & 1.661 & 0.776 & 1.262 & 3.678 & 1.97 & 1.40 & 0.45 & 0.14 & 0.03 & 0.03 & 0.01 \\
sc.120\_1500 & 30 & \textbf{0.555} & 1.462 & 2.696 & 1.049 & 1.785 & 4.546 & 109.24 & 4.20 & 1.63 & 0.43 & 0.11 & 0.20 & 0.03 \\
sc.150\_2500 & 30 & \textbf{0.809} & 1.816 & 2.850 & 1.291 & 2.050 & 5.148 & 173.13 & 6.28 & 1.99 & 0.52 & 0.18 & 0.41 & 0.05 \\
sc.200\_4000 & 30 & \textbf{0.487} & 1.285 & 2.548 & 0.939 & 1.600 & 4.245 & 209.38 & 10.47 & 3.33 & 0.89 & 0.32 & 0.89 & 0.08 \\
sc.250\_7500 & 30 & \textbf{0.300} & 0.827 & 2.591 & 0.938 & 1.421 & 4.375 & 244.92 & 17.69 & 6.87 & 1.80 & 0.53 & 2.08 & 0.14 \\
sc.300\_9000 & 30 & \textbf{0.129} & 0.477 & 2.090 & 0.609 & 0.999 & 3.667 & 257.96 & 22.04 & 8.62 & 2.31 & 0.65 & 2.69 & 0.17 \\
sc.500\_10000 & 30 & \textbf{0.199} & 1.020 & 2.613 & 0.778 & 1.325 & 4.949 & 324.28 & 34.78 & 10.02 & 2.50 & 0.84 & 2.82 & 0.19 \\
sc.1000\_20000 & 30 & \textbf{-0.280} & 0.519 & 2.446 & 0.317 & 0.912 & 5.356 & 301.56 & 99.92 & 21.96 & 5.84 & 1.78 & 6.82 & 0.39 \\
\bottomrule
\end{tabular}
}
\end{table*}

\begin{table*}
\centering
\caption{Mean/median gap to the MILP baseline (\%) and mean time (s)
by MILP termination status.}
\label{tab:strata}
\small
\setlength{\tabcolsep}{3.4pt}
\begin{tabular}{lrrrrrr}
\toprule
 & \multicolumn{3}{c}{optimal ($n{=}235$)} &
 \multicolumn{3}{c}{\shortstack{time-limited\\incumbent} ($n{=}175$)} \\
\cmidrule(lr){2-4}\cmidrule(lr){5-7}
Method & mean & med. & time & mean & med. & time \\
\midrule
MILP  & 0.00 & 0.000 & 22.76 & 0.00 & 0.000 & 309.00 \\
\midrule
L3    & \textbf{1.48} & \textbf{0.014} & 6.76 & \textbf{1.10} & \textbf{0.329} & 32.96 \\
L2    & 2.46 & 0.139 & 3.42 & 2.17 & 1.069 & 9.47 \\
L1    & 3.69 & 0.613 & 1.04 & 3.06 & 1.683 & 2.61 \\
ADSSP & 3.50 & 0.386 & 0.12 & 2.17 & 0.883 & 0.71 \\
DSSP  & 3.89 & 0.694 & 0.24 & 4.35 & 3.347 & 2.94 \\
MCF   & 8.72 & 2.332 & 0.03 & 7.75 & 6.908 & 0.17 \\
\bottomrule
\end{tabular}
\end{table*}

\begin{table*}[t]
\centering
\caption{Layer-ablation performance over all $410$ instances.
Gaps are relative to the MILP baseline on each instance; win/tie rates are relative to the best displayed heuristic.}
\label{tab:adssp_ablation}
\small
\setlength{\tabcolsep}{5pt}
\begin{tabular}{lrrrr}
\toprule
Method
& Mean gap (\%)
& Median gap (\%)
& Win/tie rate (\%)
& Mean time (s) \\
\midrule
L3              & \textbf{1.316} & \textbf{0.131} & \textbf{87.6} & 18.001 \\
L3$^-$ & 2.073          & 0.560          & 29.0          & 6.523 \\
L2              & 2.334          & 0.749          & 23.7          & 6.052 \\
L1              & 3.420          & 1.412          & 14.4          & 1.781 \\
ADSSP anchors   & 2.113          & 0.384          & 30.7          & 0.620 \\
ADSSP           & 2.891          & 0.695          & 18.3          & \textbf{0.346} \\
\bottomrule
\end{tabular}
\end{table*}

L3 attains the best objective quality consistently across dataset families, including the larger-scale instances where the MILP baseline often terminates without proving optimality. A family-wise breakdown is given in \Cref{tab:familygaps}.

The MILP status split in \Cref{tab:strata} further supports the conclusion that the advantage of L3 is not limited to instances where the MILP solves easily. L3 remains the strongest heuristic both on instances where the MILP proves optimality and on instances where the MILP terminates with only a feasible incumbent. On the $235$ instances where the MILP proves optimality, L3 has mean and median gaps of \(1.48\%\) and \(0.014\%\), respectively. On the \(175\) instances where the MILP terminates with a time-limited incumbent, L3 has mean and median gaps of \(1.10\%\) and \(0.329\%\). Thus, the proposed method performs well both when the MILP reference is certified and when exact search is unable to close the gap within the time limit.

\Cref{tab:adssp_ablation} isolates the anchor-union contribution with two ablations. Note that the win/tie rates in this table are computed over a different method pool than \Cref{tab:heuristic_ablation} of the main paper (L3, L3$^-$, L2, L1, ADSSP anchors, and ADSSP here, versus L3, L2, L1, ADSSP, DSSP, and MCF there), so L3's $87.6\%$ rate here is not directly comparable to the $90.0\%$ reported in the main paper. Running ADSSP restricted to the same anchor supports that L3 uses reaches a mean gap of $2.113\%$ with a $30.7\%$ win/tie rate, well short of L3, showing that L3 is not simply returning a good ADSSP solution. L3$^-$, which corresponds to L3 with the anchor set restricted to $\varnothing$, i.e., L3 without using any ADSSP anchors, but keeping the union search, restricted refits, and add/drop refinement, reaches a mean gap of $2.073\%$, a $0.757$-percentage-point degradation from full L3. Since $\mathcal A^+$ in \Cref{alg:anchor_union_irls} (main paper) already contains $\varnothing$, L3 is never worse than this restricted variant by construction, and is strictly better on $278/410$ instances. Together, the two ablations show that neither the ADSSP anchor supports nor the IRLS-based union search alone accounts for L3's full gain: the two mechanisms are complementary and contribute comparably. Notably, L3$^-$ already outperforms the standard ADSSP baseline itself, winning or tying head-to-head on $280/410$ instances and reducing the mean gap by $0.818$ percentage points, so L3's advantage does not depend on borrowing ADSSP's solution quality even indirectly.

\end{document}